\documentclass[leqno,11pt]{amsart}

\usepackage{verbatim}
\usepackage{amsmath,amscd,amsthm,amsxtra,amssymb}
\usepackage{epsfig,graphics,color,colortbl}
\usepackage{amssymb,latexsym}
\usepackage{mathrsfs}
\usepackage[poly,all]{xy}
\usepackage{marginnote}
\usepackage{oldgerm}
\usepackage[colorlinks=true, pdfstartview=FitV, linkcolor=blue,citecolor=blue,urlcolor=blue]{hyperref}

\usepackage[usenames,dvipsnames,svgnames,table]{xcolor}

\newtheorem{thm}{Theorem}[section]
\newtheorem{prop}[thm]{Proposition}
\newtheorem{cor}[thm]{Corollary}
\newtheorem{lem}[thm]{Lemma}

\theoremstyle{definition}
\newtheorem{defn}[thm]{Definition}
\newtheorem{Ex}[thm]{Example}

\theoremstyle{remark}
\newtheorem{Rmk}[thm]{Remark}

\numberwithin{equation}{section}

\newcommand{\Z}{\mathbb{Z}}
\newcommand{\Q}{\mathbb{Q}}

\newcommand{\A}{\mathbb{A}}

\newcommand{\g}{\mathfrak{g}}
\newcommand{\h}{\mathfrak{h}}
\newcommand{\n}{\mathfrak{n}}
\newcommand{\gl}{\mathfrak{gl}}

\newcommand{\Hom}{\mathrm{Hom}}

\newcommand{\Ht}{{\rm ht}}

\newcommand{\Ind}{{\rm Ind}}
\newcommand{\Res}{{\rm Res}}

\newcommand{\gmod}{{\rm \text{-}gmod}}
\newcommand{\proj}{\text{-}\mathrm{proj}}
\newcommand{\Mod}{\text{-}\mathrm{Mod}}

\newcommand{\rlQ}{\mathsf{Q}}   
\newcommand{\wlP}{\mathsf{P}}   
\newcommand{\cmA}{\mathsf{A}}  
\newcommand{\tf}{\tilde{f}}  
\newcommand{\te}{\tilde{e}}  
\newcommand{\cmp}{{\Delta_+}}		
\newcommand{\cmm}{{\Delta_-}}
\newcommand{\cmpm}{{\Delta_\pm}}
\newcommand{\cmn}{{\Delta_{\n}}}
\newcommand{\wt}{\mathrm{wt}} 		
\newcommand{\ep}{\varepsilon}  		
\newcommand{\ph}{\varphi}  		
\newcommand\Aq[1][{\mathfrak g}]{U_q^{-}(#1)^\vee} 
\newcommand\Aqq[1][{\mathfrak g}]{U_{\A}^{-}(#1)^\vee}
\newcommand{\ait}{\mathsf{t}}  		

\newcommand{\ST}{\mathsf{ST}}   
\newcommand{\res}{\mathsf{res}}   

\newcommand{\bR}{\mathbf{k}}   

\newcommand{\conv}{{\mathbin{\scalebox{1.1}{$\mspace{1.5mu}\circ\mspace{1.5mu}$}}}}
\newcommand{\hconv}{\mathbin{\scalebox{.9}{$\nabla$}}}
\newcommand{\sconv}{\mathbin{\scalebox{.9}{$\Delta$}}}

\renewcommand{\Im}{\operatorname{Im}}

\newcommand{\hd}{{\mathrm{hd}}}      					 

\newcommand{\fc}[1][\lambda]{{\Phi_{#1}}}  	

\newcommand{\dfc}[1][\lambda]{{\Psi_{#1}}}  	

\newcommand{\prj}{{\mathsf{p}}}  	
\newcommand{\inj}{{\mathsf{i}}}  	
\newcommand{\mRes}[1][\lambda]{{\Res_{#1}}}  	

\newcommand{\lan}{\langle} 	
\newcommand{\ran}{\rangle}	

\newcommand{\rmat}[1]{{\mathbf{r}}_%
{\mspace{-2mu}\raisebox{-.6ex}{${\scriptstyle{#1}}$}}}
\newcommand{\Dd}{\text{ \textfrak{d}}} 			

\newcommand{\Par}{\mathcal{P}}					
\newcommand{\SST}{\mathsf{SST}}					
\newcommand{\Sp}{\mathcal{S}}					
\newcommand{\qch}{\mathrm{ch}_q}					

\newcommand{\low}{{\rm low}}
\newcommand{\up}{{\rm up}}
\newcommand{\one}{{\bf 1}}
\newcommand{\sh}{\mathsf{sh}}

\begin{document}

\title[A remark on convolution products for quiver Hecke algebras]
{A remark on convolution products for quiver Hecke algebras}

\author[Myungho Kim]{Myungho Kim$^1$}
\thanks{$^1$ This work was supported by the National Research Foundation of
Korea(NRF) Grant funded by the Korea government(MSIP) (NRF-2017R1C1B2007824).}
\address{Department of Mathematics, Kyung Hee University, Seoul 02447, Korea }
\email{mkim@khu.ac.kr}

\author[Euiyong Park]{Euiyong Park$^2$}
\thanks{$^2$ This work was supported by the 2015 Research Fund of the University of Seoul.}
\address{Department of Mathematics, University of Seoul, Seoul 02504, Korea}
\email{epark@uos.ac.kr}


\begin{abstract}

In this paper, we investigate a connection between convolution products for quiver Hecke algebras and tensor products for quantum groups.
We give a categorification of the natural projection
$
\pi_{\lambda, \mu} : V_\A(\lambda)^\vee \otimes_{ \A} V_\A(\mu)^\vee \twoheadrightarrow V_\A(\lambda+ \mu)^\vee
$
sending the tensor product of the highest weight vectors to the highest weight vector in terms of convolution products.
When the quiver Hecke algebra is symmetric and the base field is of characteristic $0$, 
 we obtain a positivity condition on some coefficients associated with the projection $\pi_{\lambda, \mu}$ and the upper global basis,
and prove several results related to the crystal bases. We then apply our results to finite type $A$ using the homogeneous simple modules $\Sp^T$ indexed by one-column tableaux $T$.
\end{abstract}

\maketitle


\vskip 2em

\section*{Introduction}

\emph{Quiver Hecke algebras} (or \emph{Khovanov-Lauda-Rouqiuer algebras}) were introduced to give a categorification of the half of a quantum group \cite{KL09, KL11,R08}.
These algebras have special graded quotients, called \emph{cyclotomic quiver Hecke algebras}, which categorify irreducible highest weight integrable modules \cite{KK11}. 
Cyclotomic quiver Hecke algebras of affine type $A$ are isomorphic to blocks of cyclotomic Hecke algebras \cite{BK09, R08}. 
In this sense, quiver Hecke algebras are a vast generalization of Hecke algebras in the direction of categorification. 
When quiver Hecke algebras are symmetric and the base field is of characteristic 0,
they have a geometric realization using quiver varieties which yields that 
the \emph{upper global basis} (\emph{dual canonical basis}) corresponds to the isomorphism classes of self-dual simple modules  \cite{R11,VV11}, 
and they also give a monoidal categorification for some quantum cluster algebras \cite{KKKO17}.
The study of quiver Hecke algebras has been one of the main research themes on quantum groups in a viewpoint of categorification.

Let $U_q(\g)$ be the quantum group associated with a generalized Cartan matrix $\cmA$ and $R(\beta)$ a quiver Hecke algebra associated with $\cmA$ and 
$\beta \in \rlQ_+$, where $\rlQ_+$ is the positive cone of the root lattice $\rlQ$.  
The convolution product $\circ$ and the restriction functor $\Res$ give twisted bialgebra structures on 
the Grothendieck groups $[R\proj]$ and $[R\gmod]$, where  $R\proj$ (resp.\  $R\gmod$) is the category of finitely generated projective 
(resp.\  finite-dimensional) graded $R$-modules. It was proved that there are bialgebra isomorphisms $U_\A^-(\g) \simeq [R\proj]$ and $U_\A^-(\g)^\vee \simeq [R\gmod]$, 
where $U_\A^-(\g)^\vee $ is the dual of $U_\A^-(\g)$.
Similarly, cyclotomic quiver Hecke algebras give a categorification for irreducible highest weight modules. 
For a dominant integral weight $\Lambda$, let $V_q(\Lambda)$ be the irreducible highest weight $U_q(\g)$-module with highest weight $\Lambda$ and 
$R^\Lambda(\beta) $ the cyclotomic quiver Hecke algebras of $R(\beta)$ corresponding to $\Lambda$. 
It was proved that there exist $U_\A(\g)$-module isomorphisms $V_\A(\Lambda) \simeq [R^\Lambda\proj]$ and $V_\A(\Lambda)^\vee \simeq [R^\Lambda\gmod]$.
Here $V_\A(\Lambda)^\vee$ is the dual of $V_\A(\Lambda)$, and
$R^\Lambda\proj$ (resp.\  $R^\Lambda\gmod$) is the category of finitely generated projective (resp.\  finite-dimensional) graded $R^\Lambda$-modules.

In this paper, we investigate a connection between convolution products for quiver Hecke algebras and tensor products for quantum groups. 
For dominant integral weights $\lambda$ and $\mu$, we consider a chain of homomorphisms
$$
U_\A^-(\g)  \buildrel \fc \over \rightarrowtail  U_\A^-(\g) \otimes U_\A^-(\g)  \buildrel {  \prj_{\lambda} \otimes \prj_{\mu}   } \over \twoheadrightarrow  V_{\A}(\lambda) \otimes_\A V_\A(\mu), 
$$
where $\fc$ is defined in $\eqref{Eq: fc}$ and $\prj_{\lambda}$ is the natural projection defined in $\eqref{Eq: proj_A}$.
We interpret the above homomorphisms in terms of the restriction functor $\Res$ in a viewpoint of categorification, which yields our main result, that is
a categorification of the surjective homomorphism
\begin{align*}
\pi_{\lambda, \mu} : V_\A(\lambda)^\vee \otimes_{ \A} V_\A(\mu)^\vee \twoheadrightarrow V_\A(\lambda+ \mu)^\vee ,
\end{align*}
where $\pi_{\lambda, \mu}$ is the natural projection sending the tensor product of the highest weight vectors to the highest weight vector.
More precisely, we prove that
\[
 \pi_{\lambda, \mu} ( [M] \otimes [N] ) = q^{(\beta_2, \lambda)}[M\conv N] 
\] 
for $M \in R^\lambda(\beta_1)\gmod$ and $N \in R^\mu(\beta_2)\gmod$ (Corollary \ref{Cor: main}). 
For types $A_\infty$ and $A_n^{(1)}$, the same connection was investigated in terms of affine Hecke algebras and upper global bases 
 in \cite{LNT03}. Note also that a categorification of tensor products of irreducible highest weight modules was studied in \cite{Web17}.

When $R$ is \emph{symmetric}, we study further with upper global bases and \emph{crystal bases}. 
We assume that $R$ is symmetric and the base field is of characteristic 0. 
Let $B(\lambda)$ be a crystal of $V_q(\lambda)$ and $b_\lambda$ the highest weight vector of $B(\lambda)$. 
For $b\in B(\lambda)$, we denote by $G^{\up}_\lambda(b)$ the member of 
the upper global base $\mathbf B^\up (\lambda)$ corresponding to $b$ and by $L(b)$  the self-dual simple $R$-module corresponding to $b$.
For $b \in B(\lambda)$ and $ b' \in B(\mu)$,  if we write   
$$
  \pi_{\lambda, \mu} ( G^{\up}_\lambda(b) \otimes G^{\up}_\mu(b') ) = \sum_{b'' \in B(\lambda+\mu)} A_{b''}(q)   G_{ \lambda+\mu  }^{\up}(b'') 
$$ 
for $A_{b''}(q) \in  \Q(q) $, then we prove  that 
$$
A_{b''}(q) = q^{(\beta_2, \lambda)} [L(b)\conv L(b') : L(b'')]_q ,
$$
 which yields the positivity $A_{b''}(q) \in \Z_{\ge 0}[q^\pm]$ (Corollary \ref{Cor: positivity}).
Let $C_{\lambda_{1}, \ldots, \lambda_r }$ be the connected component of $B(\lambda_1)\otimes \cdots \otimes B(\lambda_r) $ containing $b_{\lambda_1}\otimes \cdots \otimes b_{\lambda_r}$.
We then prove that  if 
\begin{enumerate}
\item[(i)] $b_1 \otimes \cdots \otimes b_r \in C_{\lambda_{1}, \ldots, \lambda_r }$,
\item[(ii)] $\hd( L(b_1) \conv \cdots \conv L(b_r) )$ is simple,
\end{enumerate}
then 
\[
   \, \hd( L(b_1) \conv \cdots \conv L(b_r) ) \simeq L( b_1 \otimes \cdots \otimes b_r )
\]
up to a grading shift (Theorem \ref{Thm: app_tensor}). When applying this result to a pair of real simple $R$-modules, we have more results.
Suppose that $ L(b_1)$ and $L(b_2)$ are real.  We then compute the degree $\Lambda(L(b_1), L(b_2))$ of the R-matrix and related quantities
 in terms of weights and dominant integral weights (Corollary \ref{Cor: app_tensor}). 
Moreover, if there is $b_1'\otimes b_2' \in C_{\lambda_1', \lambda_2'}$ such that $\lambda_1 + \lambda_2 = \lambda_1' + \lambda_2'$ and $b_1\otimes b_2$ is crystal equivalent to $b_1' \otimes b_2'$, then 
we show that 
$$
\Hom_{R\gmod} ( q^d L(b_1) \conv  L(b_2),  L(b'_2) \conv L(b'_1))  \simeq {\mathbf k},
$$
where $d$ is given in Theorem \ref{Thm: sR}. As a consequence, we have a couple of equivalent conditions to the condition that $L(b_1)$ and $L(b_2)$ strongly commute (Corollary \ref{Cor: app2}).  

As an application, we apply our results to finite type $A$. We compute several examples with the homogeneous simple $R^{\Lambda_i}$-modules $\Sp^T$ indexed by one-column tableaux $T$,
which were introduced in \cite{BK09,KR10}.
In Theorem \ref{Thm: KL poly}, we also explain that the graded decomposition number between $\Sp^{T_r} \conv \cdots \conv \Sp^{T_1} $ and simple $R$-modules 
in terms of the Kazhdan-Lusztig polynomials by using the results on the entries of  the transition matrix between the standard monomial and upper global bases of the irreducible $U_q(\gl_n)$-modules 
given in \cite{Brun06}. Note that these entries also appear in composition multiplicities of the standard modules for the finite $W$-algebras/shifted Yangians \cite{BK08}.

The paper is organized as follows. In Section \ref{Sec: quantum groups}, we review global bases and crystal bases of quantum groups. In Section \ref{Sec: quiver Hecke algebra},
we recall the categorification using quiver Hecke algebras. In Section \ref{Sec: tensor and convolution}, we investigate a connection between tensor products and convolution products, and provide the main theorems. In Section \ref{Sec: application}, we apply our main results to finite type $A$.

\vskip 2em

\section{Quantum groups}  \label{Sec: quantum groups}
\subsection{Quantum groups}
 Let $I$ be an index set.
A {\it Cartan datum} $ (\cmA,\wlP,\Pi,\wlP^\vee,\Pi^\vee) $
consists of
\begin{enumerate}
\item a matrix $\cmA=(a_{ij})_{i,j\in I}$, called the \emph{symmetrizable generalized Cartan matrix}, satisfying
\begin{enumerate}
\item $a_{ii}=2$ for $i \in I$  and $a_{ij} \in \Z_{\le 0}$ for $i \neq j$,
\item $a_{ij} = 0$ if and only if $a_{ji} = 0$,
\item there exists a diagonal matrix $D={\rm diag}(\mathsf d_i \mid i \in I)$ such that $D\cmA$ is symmetric, and $\mathsf d_i$'s are relatively prime positive integers,
\end{enumerate}
\item a free abelian group $\wlP$, called the {\em weight lattice},
\item $\Pi = \{ \alpha_i \mid i\in I \} \subset \wlP$,
called the set of {\em simple roots},
\item $\wlP^{\vee}=
\Hom_{\Z}( \wlP, \Z )$, called the \emph{coweight lattice},
\item $\Pi^{\vee} =\{ h_i \in \wlP^\vee \mid i\in I\}$, called the set of {\em simple coroots}, satisfying
\begin{enumerate}
\item $\lan h_i, \alpha_j \ran = a_{ij}$ for $i,j \in I$,
\item $\Pi$ is linearly independent over $\Q$,
\item for each $i\in I$, there exists $\Lambda_i \in \wlP$, called the \emph{fundamental weight}, such that $\lan h_j,\Lambda_i \ran =\delta_{j,i}$ for all $j \in I$.
\end{enumerate}
\end{enumerate}

Set $\mathfrak h=\Q \otimes_\Z \wlP$. We fix a nondegenerate symmetric bilinear form $( \cdot \, , \cdot )$ on $\mathfrak h^*$ satisfying 
\begin{equation} \label{eq:bilinear}
(\alpha_i,\alpha_j)=\mathsf d_i a_{ij} \quad (i,j \in I),\quad \text{and } \quad  \lan h_i,  \lambda\ran = \dfrac{2 (\alpha_i,\lambda)}{(\alpha_i,\alpha_i)} \quad (\lambda \in \mathfrak h^*, \ i \in I).
\end{equation}

Set 
$$
 \rlQ = \bigoplus_{i \in I} \Z \alpha_i, \quad \rlQ_+ = \sum_{i\in I} \Z_{\ge 0} \alpha_i, \quad \rlQ_- = -\rlQ_+.
$$
We call $\rlQ$ (repectively, $\rlQ_+$) the \emph{root lattice} (respectively, the \emph{positive root lattice}).
We define $\Ht (\beta)=\sum_{i \in I} k_i $  for $\beta=\sum_{i \in I} k_i \alpha_i \in \rlQ_+$,
and denote by $\wlP^+: =\{ \lambda \in \wlP \mid \lan h_i, \lambda\ran \ge 0 \ \text{for all }  \ i \in I \}$ the set of \emph{dominant integral weights}. 

Let $U_q(\g)$ be the quantum group associated with the Cartan datum $(\cmA, \wlP,\wlP^\vee \Pi, \Pi^{\vee})$, which
is a $\Q(q)$-algebra generated by $f_i$, $e_i$ $(i\in I)$ and $q^h$ $(h\in \wlP)$ with certain defining relations (see \cite[Chater 3]{HK02} for details).
Let $U_q(\g)_\beta = \{ x \in U_q(\g) \mid q^h xq^{-h} = q^{\lan h, \beta \ran} x \quad \text{for all} \ h \in \wlP^\vee\}$. Then we have
$U_q(\g) =\bigoplus_{\beta \in \rlQ} U_q(\g)_\beta$  and 
$U_q^-(\g) =\bigoplus_{\beta \in \rlQ_-} U_q(\g)_\beta$ .
In the sequel,  for a $\wlP$-graded vector space $V=\bigoplus_{\mu \in \wlP} V_{\mu}$, we denote 
by 
$$
V^\vee:=\bigoplus_{\mu \in \wlP} \Hom_{\Q(q)}(V_\mu, \Q(q))
$$ 
the {\it restricted dual} of $V$.
 We often write $\langle v, f \rangle = f(v)$ for $f\in V^\vee$ and $v \in V$ in pairing notation.  
The $\Q(q)$-algebra homomorphisms (comultiplications) $\cmpm: U_q(\g) \longrightarrow U_q(\g) \otimes_{\Q(q)} U_q(\g) $ are defined by
\begin{align*}
\cmp(f_i) &= f_i \otimes t_i^{-1}+ 1 \otimes f_i, \quad \cmp(e_i) = e_i \otimes 1 + t_i \otimes e_i, 
 \quad \ \ \cmp(q^h)=q^h \otimes q^h, 
\\
\cmm(f_i) &= f_i \otimes 1 + t_i \otimes f_i,\ \ \quad \cmm(e_i) = e_i \otimes t_i^{-1} + 1 \otimes e_i,
 \quad \cmm(q^h)=q^h \otimes q^h, 
\end{align*}
where $t_i = q^{\mathsf{d}_i  h_i }$.

For homogeneous elements  $x,y,z,w \in U_q^-(\g)$, we define 
\[
(x\otimes y)  \cdot  (z \otimes w) = q^{- (\wt(y), \wt(z))} xz \otimes yw,
\]
which gives   another  $\Q(q)$-algebra structure on $U_q^-(\g) \otimes_{\Q(q)}U_q^{-}(\g)$.
 Then we have a $\Q(q)$-algebra homomorphism $\cmn:U_q^-(\g)\to U_q^-(\g)\otimes_{\Q(q)}U_q^-(\g)$ with respect to the above multiplication, which is given by
\[
\cmn(f_i)= 1 \otimes f_i + f_i \otimes 1 \quad \text{for } \ i  \in I. 
\]
 We often write $\cmn(x) = x_{(1)}\otimes x_{(2)} $ for $x \in U_q^-(\g)$ in Sweedler's notation.
We set 
\[
\Aq :=  \bigoplus_{\beta \in \rlQ^-} \Aq_\beta,
\]  
and define, for $f,g \in \Aq$ and $u\in U_q^-(\g)$,
\[
(f \cdot g)(u) := (f \otimes g)( \cmn(u) ) = f(u_{(1)}) f(u_{(2)}), \quad \text{ where } \cmn(u) =  u_{(1)} \otimes u_{(2)}.
\]
Let $\A=\Z[q,q^{-1}]$, and $U_\A^-(\g)$ denote the subalgebra of $U_q^-(\g)$ generated by $f_i^{(n)} := f_i^n / [n]_i!$ for $i\in I$ and $n\in \Z_{\ge0}$,
where
\begin{align*}
q_i = q^{ \mathsf d_i }, \quad [n]_i =\frac{ q^n_{i} - q^{-n}_{i} }{ q_{i} - q^{-1}_{i} },\quad  [n]_i! = \prod^{n}_{k=1} [k]_i.
\end{align*}
Let 
\[
\Aqq := \{ f \in  \Aq \mid f( U_\A^-(\g)) \subset \A )\},
\]
which forms an $\A$-subalgebra of $\Aq$.

A $U_q(\g)$-module is called \emph{integrable} if $M=\bigoplus_{\lambda \in \wlP} M_\lambda$, where $M_\lambda:=\{ v \in M \mid q^h v = q^{\lan h, \lambda \ran} v \quad \text{for all} \ h \in \wlP^\vee \}$, $\dim_{\Q(q)} M_\mu < \infty$, and the actions of $e_i$ and $f_i$ on $M$ are locally nilpotent for all $i \in I$. We denote by $\mathcal O_{\rm int}(\g)$ the category of integrable left $U_q(\g)$-modules $M$ such that there exist finitely many $\lambda_1,\ldots,\lambda_m$ with $\wt(M) \subset \bigcup_{j} (\lambda_j-\rlQ_+)$.
Let $\ait$ be the $\Q(q)$-algebra anti-involution on $U_q(\g)$ given by $\ait(f_i) = e_i$, $\ait(e_i) = f_i$ and $\ait(q^h) = q^h$ for $i\in I$ and $h \in \wlP^\vee$.
Then  the restricted dual  $M^\vee$ of an integrable module $M$ has the left $U_q(\g)$-module structure defined by
\[
(xf)(m) := f( \ait(x) m ) \quad \text{ for $f\in M^\vee$, $m \in M$ and $x \in U_q(\g)$}.
\] 
Note that, using pairing notation, we have $\langle  m,  xf \rangle = \langle \ait(x) m, f \rangle$. 
For $\lambda \in \wlP^+$, let $V_q(\lambda)$ be the irreducible highest weight module of highest weight $\lambda$, and 
let $u_\lambda$ denote a highest weight vector of $V_q(\lambda)$. 
Note that, as $U_q(\g)$-modules, 
 \begin{align} \label{Eq: dual}
V_q(\lambda) \simeq V_q(\lambda)^\vee.
\end{align} 

We set  
\begin{align*}
V_\A(\lambda) := U_\A(\g) u_\lambda, \qquad 
V_\A(\lambda)^\vee   :=  \{ \phi \in V_q(\lambda)^\vee  \mid \phi( V_\A(\lambda)  \subset \A)  \}.
\end{align*}

Consider the natural projection map $\prj_\lambda: U_q^-(\g) \twoheadrightarrow V_q(\lambda)$ given by $\prj_\lambda(x) = x u_\lambda$. 
Restricting to  $U_\A^-(\g)$, we have
\begin{align} \label{Eq: proj_A}
\prj_\lambda|_{U_\A^-(\g)}: U_\A^-(\g) \twoheadrightarrow V_\A(\lambda).
\end{align}
We simply write $\prj_\lambda$ for $\prj_\lambda|_{U_\A^-(\g)}$. Restricting the dual map $  \inj _\lambda:=\prj_\lambda^\vee  : V_q(\lambda)^{  \vee  }  \rightarrowtail U_q^-(\g)^{  \vee  } $
to  $V_\A(\lambda)^\vee$, we have the following 
\begin{align} \label{Eq: Inj_A}
\inj_\lambda: V_\A(\lambda)^\vee \rightarrowtail \Aqq \subset  U_q^-(\g)^\vee.
\end{align}

For $U_q(\g)$-modules $M$ and $N$, we define $U_q(\g)$-module actions on $M \otimes_{\Q(q)} N$ by
\[
x \cdot (m\otimes n) := \Delta_{\pm}(x) (m \otimes n) \quad \text{for $x \in U_q(\g)$ and $m \otimes n \in M \otimes_{\Q(q)} N$,}
\]
which is denoted by $ M \otimes_\pm N$ as a $U_q(\g)$-module.
Note that, as a $U_q(\g)$-module,
\begin{align} \label{eq:MtensNdual}
M^\vee \otimes_{\mp} N^\vee \simeq (M \otimes_\pm N)^\vee,
\end{align}
where the isomorphisms are given by
\begin{align*} 
\lan f\otimes g, v\otimes w\ran := f(v)g(w)
\qquad  \text{for} \ f\in M^\vee, \, g\in N^\vee, \, v\in M, \text{and} \ w\in N. 
\end{align*}

\subsection{Global bases and crystal bases}
Let us recall the notions of global bases and crystal bases briefly (see \cite{Kas91, Kas93}, \cite[Chap.\ 4 and 6]{HK02} for details).
Let $\A_0$ be the subring of $\Q(q)$ consisting of rational functions which are regular at $q=0$.
For each $i \in I$, let us denote by $\tf_i$ and $\te_i$ the \emph{Kashiwara operators}  on $U_q^-(\g)$ given in \cite[(3.5.1)]{Kas91}, and 
set
\begin{align*}
&L(\infty) : = \sum_{l \in \Z_{\ge0}, \ i_1,\ldots, i_l \in I} \A_0 \tf_{i_1} \cdots \tf_{i_l} \one \subset U_q^-(\g),  \ \  \overline{L(\infty)} : = \{ \overline{x} \in U_q^-(\g) \mid x \in L(\infty) \}, \\
&B(\infty) : = \{ \tf_{i_1} \cdots \tf_{i_l} \one \mod qL(\infty) \mid l \in \Z_{\ge0}, \ i_1,\ldots, i_l \in I \} 
\subset L(\infty) /qL(\infty),
\end{align*}
where
$^- : U_q(\g) \buildrel \sim \over \rightarrow  U_q(\g)$ is the $\Q$-algerba automorphism given by $\overline{e_i}=e_i, \ \overline{f_i}=f_i, \ \overline{q^h}=q^{-h},\ \text{and} \  \overline{q}=q^{-1}$.
Then the triple $((\Q\otimes_\Z U_\A^-(\g)), L(\infty), \overline{L(\infty)})$ is \emph{balanced}; i.e.,  
the natural map $(\Q\otimes_\Z U_\A^-(\g)) \cap L(\infty) \cap \overline{L(\infty)} \to L(\infty)/qL(\infty)$ is a $\Q$-linear isomorphism. 
Denote the inverse  by $G^\low$. The set 
$$\mathbf B^\low (\infty) : =\{G^\low(b) \in U_\A^-(\g) \mid b \in B(\infty) \}$$
forms an $\A$-basis of $U_\A^-(\g)$ and called the \emph{lower global basis} of $U_q^-(\g)$. Then
we have the dual basis
$$\mathbf B^\up (\infty) : =\{G^\up(b)  \in U_\A^-(\g)^\vee \mid b \in B(\infty) \},$$
which forms an $\A$-basis of $U_\A^-(\g)^\vee$ called the  \emph{upper global basis} of $U_q^-(\g)$;
i.e., $\lan G^\low(b), G^\up(b')\ran = \delta_{b,b'}$ for $b, b' \in B(\infty).$

Let $\lambda \in \wlP^+$ and set
\begin{align*}
&L^\low(\lambda) : = \sum_{l \in \Z_{\ge0}, \ i_1,\ldots, i_l \in I} \A_0 \tf_{i_1} \cdots \tf_{i_l} u_\lambda \subset V_q(\lambda),  \quad  \overline{L^\low(\lambda)} : = \{ \overline{v} \in V_q(\lambda) \mid v \in L^\low(\lambda) \} \\
&B(\lambda) : = \{ \tf_{i_1} \cdots \tf_{i_l} u_\lambda \mod qL^\low(\lambda) \mid l \in \Z_{\ge0}, \ i_1,\ldots, i_l \in I \} 
\subset L^\low(\lambda) /qL(\lambda),
\end{align*}
where $^- : V_q(\lambda) \to V_q(\lambda)$ is the $\Q$-linear automorphism given by $\overline{xu_\lambda}:=\overline{x} u_\lambda$.
Denote the inverse of the $\Q$-linear isomorphism $(\Q\otimes_\Z V_\A(\lambda)) \cap L^\low (\lambda) \cap \overline{L^\low(\lambda)} \to L^\low(\lambda)/qL^\low(\lambda)$  by $G^\low_\lambda$. Then the set 
$$\mathbf B^\low (\lambda) : =\{G_\lambda^\low(b) \mid b \in B(\lambda) \}$$
forms an $\A$-basis of $V_\A(\lambda)$ called the \emph{lower global basis} of $V_q(\lambda)$. Similarly, we have
the dual basis
$$\mathbf B^\up (\lambda) : =\{G^\up_\lambda(b) \in V_\A(\lambda)^\vee \mid b \in B(\lambda) \},$$
which forms an $\A$-basis of $V_\A(\lambda)^\vee$ called the  \emph{upper global basis} of $V_q(\lambda)$; i.e.,
$\lan G_\lambda^\low(b), G_\lambda^\up(b')\ran = \delta_{b,b'}$ for $b, b' \in B(\lambda)$.
Set $L^{up} (\lambda):=\{f \in V_q(\lambda)^\vee \mid \lan f, L^\low(\lambda) \ran  \subset \A_0\}$.

The global basis of $U_q^-(\g)$ and that of $V_q(\lambda)$ are compatible in the following sense:
\begin{thm} {\rm(\cite{Kas91})} We have
$$\prj_\lambda L(\infty) = L^\low(\lambda). $$
The induced surjectve map $ \overline{\prj}_\lambda:L(\infty) / qL(\infty) \to L^\low(\lambda) /  qL^\low(\lambda)$ is a bijection between $\{b \in B(\infty) \mid \overline{\prj}_\lambda(b) \neq 0\}$ and $B(\lambda)$.
If $\overline{\prj}_\lambda (b) \neq 0$, then 
$\prj_\lambda G^\low(b)= G^\low(b) u_\lambda = G_\lambda^\low(\overline{\prj}_\lambda(b))$.
Taking duals,  the map $\inj_\lambda$ induces an injective $\A_0$-linear map from $L^\up(\lambda)$ to $L^\up(\infty)$ and
for $b \in B(\lambda)$ we have
$$\inj_\lambda G_\lambda^\up(b)=  G^\up(\overline{\inj}_\lambda(b)),$$
where $\overline{\inj}_\lambda$ denotes the inverse of $\overline{\prj}_\lambda$ on $B(\lambda)$.
\end{thm}
 In the sequel, we will omit the maps $\inj_\lambda$, $\overline{\inj}_\lambda$ and identify  $G^\up(\overline{\inj}_\lambda(b))$ with $G_\lambda^\up(b)$  if there is no afraid of confusion.
 \medskip

For $M \in \mathcal O_{\rm int} (\g)$, a pair $(L,B)$ of 
a free $\A_0$-submodule $L$ of $M$ and a $\Q$-basis $B$ of $L/qL$  is called a \emph{lower crystal base of $M$} if $L$ and $B$ are stable under the Kashiwara operator $\te_i, \tf_i$ with some conditions (see \cite[Chap.\ 4]{HK02}).
The set $B$ equips with an $I$-colored oriented graph structure by setting $b \buildrel i \over \rightarrow  \tf_ib$ for $b, \tf_i(b) \in B$.
When $(L_1,B_1)$ and $(L_2,B_2)$ are crystal bases of $M_1$ and $M_2$ in $\mathcal O_{\rm int} (\g)$, the pair $(L_1\otimes_{\A_0} L_2, B_1\otimes B_2)$ is a crystal basis of $M_1\otimes_- M_2$ , where $B_1 \otimes B_2 := \{b_1 \otimes b_2 \mid b_i \in B_i, \ (i=1,2) \} \subset L_1\otimes_{\A_0} L_2 / q(L_1\otimes_{\A_0} L_2)$ (\cite{Kas91}).
The Kashiwara operator on $B_1\otimes B_2$ can be described  explicitly as  follows.
\begin{defn} 
For $i\in I$, $b_1 \in B_1$ and $b_2 \in B_2$, 
\begin{align*} 
\wt(b_1 \otimes b_2) &= \wt(b_1) + \wt(b_2), \\
\ep_i(b_1 \otimes b_2) &= \max(  \ep_i(b_1), \ep_i(b_2) - \langle h_i, \wt(b_1) \rangle ),\\
\ph_i(b_1 \otimes b_2) &= \max(  \ph_i(b_1)+  \langle h_i, \wt(b_2) \rangle, \ph_i(b_2) ),\\
\te_i(b_1 \otimes b_2) &= 
\left\{
\begin{array}{ll}
\te_i(b_1) \otimes b_2 & \text{ if } \ph_i(b_1) \ge \ep_i(b_2),\\
b_1 \otimes \te_i(b_2) & \text{ if } \ph_i(b_1) < \ep_i(b_2),
\end{array}
\right.\\
\tf_i(b_1 \otimes b_2) &= 
\left\{
\begin{array}{ll}
\tf_i(b_1) \otimes b_2 & \text{ if } \ph_i(b_1) > \ep_i(b_2),\\
b_1 \otimes \tf_i(b_2) & \text{ if } \ph_i(b_1) \le \ep_i(b_2).
\end{array}
\right.
\end{align*}
\end{defn}

In particular, the pair 
$(L^\low(\lambda_1) \otimes_{\A_0} \cdots \otimes_{\A_0}L^\low(\lambda_r), B(\lambda_1)\otimes \cdots \otimes B(\lambda_r))$ is a crystal basis of 
 $V_q(\lambda_1)\otimes_-\cdots \otimes_- V_q(\lambda_r)$ 
for $\lambda_k \in \wlP^+$  $  (k=1,\ldots,r)  $.
We set
\begin{align*}
\text{$C_{\lambda_{1}, \ldots, \lambda_r }$ := the connected component of $B(\lambda_1)\otimes \cdots \otimes B(\lambda_r)$ containing $b_{\lambda_1}\otimes \cdots \otimes b_{\lambda_r} $.} 
\end{align*}

Let $B, B'$ be crystals. 
For $b\in B$ and $b' \in B'$, we say that \emph{$b$ is crystal equivalent to $b'$}, which is denoted by 
$b \simeq b'$,  if there exists an isomorphism between $C(b)$ and $C(b')$ sending $b$ to $b'$,
where $C(b)$ and $C(b')$ are the connected components of $B$ and $B'$ containing $b$ and $b'$ respectively.

We recall the following proposition  (\cite{Kas90, Kas91}, \cite[Proposition 3.29]{Kimura12}), which plays an important role later. 
\begin{prop}   \label{Prop: global}
Let $\lambda_k \in \wlP^+$ and $b_k \in B(\lambda_k)$ for $k=1, \ldots, r$, and set $\lambda = \sum_{k=1}^r \lambda_k$.
If $b_1 \otimes \cdots \otimes b_r \in C_{\lambda_{1}, \ldots, \lambda_r} $, then 
 \[
\pi_{\lambda_{1}, \ldots, \lambda_r} (  G^{\up}_{\lambda_1}(b_1) \otimes \cdots \otimes  G^{\up}_{\lambda_r}(b_r) )
\equiv  G^{\up}_{\lambda}(b_1 \otimes \cdots \otimes  b_r) \qquad {\mathrm{mod} \ \ qL^{\up} (\lambda), }
\]
where $\pi_{\lambda_{1}, \ldots, \lambda_r}: V_q(\lambda_1)^\vee\otimes_+ \cdots \otimes_+ V_q(\lambda_r)^\vee \to V_q(\lambda)^\vee $ is the dual map of the $U_q(\g)$-module homomorphism $\iota_{\lambda_{1}, \ldots, \lambda_r}: V_q(\lambda)\to V_q(\lambda_1)\otimes_-\cdots \otimes_-V_q(\lambda_r)$ which is given by
$\iota_{\lambda_{1}, \ldots, \lambda_r}(u_{\lambda})=(u_{\lambda_1}\otimes \cdots \otimes u_{\lambda_r}).$
\end{prop}

\vskip 2em 

\section{Quiver Hecke algebras}  \label{Sec: quiver Hecke algebra}
\subsection{Quiver Hecke algebras}

Let $\mathbf k$ be a field and let $(\cmA,\wlP, \Pi,\wlP^{\vee},\Pi^{\vee})$ be a Cartan datum.
 Let us take  a family of polynomials $(Q_{i,j})_{i,j\in I}$ in ${\mathbf k}[u,v]$
which satisfies
\begin{enumerate}
\item $Q_{i,j}(u,v)=Q_{j,i}(v,u)$ for any $i,j\in I$,
\item $Q_{i,j}(u,v)=0$ for $i=j\in I$,
\item for $i\not=j$, we have
$$
Q_{i,j}(u,v) =
\sum\limits_{ \substack{ (p,q)\in \Z^2_{\ge0} 
\\ (\alpha_i , \alpha_i)p+(\alpha_j , \alpha_j)q=-2(\alpha_i , \alpha_j)}}
t_{i,j;p,q} u^p v^q
$$
with $t_{i,j;p,q}\in{\mathbf k}$, $t_{i,j;p,q}=t_{j,i;q,p}$ and $t_{i,j:-a_{ij},0} \in {\mathbf k}^{\times}$, 
where $(\ , \ )$ is the symmetric bilinear form in \eqref{eq:bilinear}.
\end{enumerate}
We denote by
${\mathfrak{S}}_{n} = \langle s_1, \ldots, s_{n-1} \rangle$ the symmetric group
on $n$ letters, where $s_i= (i, i+1)$ is the transposition of $i$ and $i+1$.
Then ${\mathfrak{S}}_n$ acts on $I^n$ by place permutations.
For $n \in \Z_{\ge 0}$ and $\beta \in \rlQ_+$ such that $\Ht(\beta) = n$, we set
$$I^{\beta} = \{\nu = (\nu_1, \ldots, \nu_n) \in I^{n}\mid \alpha_{\nu_1} + \cdots + \alpha_{\nu_n} = \beta\}.$$

For $\beta \in \rlQ_+$ with $\Ht(\beta)=n$, the \emph{quiver Hecke algebra}  
$R(\beta)$  at $\beta$ associated
with a matrix
$(Q_{i,j})_{i,j \in I}$ is the ${\mathbf k}$-algebra generated by
the elements $\{ e(\nu) \}_{\nu \in  I^{\beta}}$, $ \{x_k \}_{1 \le
k \le n}$, $\{ \tau_m \}_{1 \le m \le n-1}$ satisfying 
\begin{align*} 
& e(\nu) e(\nu') = \delta_{\nu, \nu'} e(\nu), \ \
\sum_{\nu \in  I^{\beta} } e(\nu) = 1, \allowdisplaybreaks\\
& x_{k} x_{m} = x_{m} x_{k}, \ \ x_{k} e(\nu) = e(\nu) x_{k}, \allowdisplaybreaks\\
& \tau_{m} e(\nu) = e(s_{m}(\nu)) \tau_{m}, \ \ \tau_{k} \tau_{m} =
\tau_{m} \tau_{k} \ \ \text{if} \ |k-m|>1, \allowdisplaybreaks\\
& \tau_{k}^2 e(\nu) = Q_{\nu_{k}, \nu_{k+1}} (x_{k}, x_{k+1})
e(\nu), \allowdisplaybreaks\\
& (\tau_{k} x_{m} - x_{s_k(m)} \tau_{k}) e(\nu) = \begin{cases}
-e(\nu) \ \ & \text{if} \ m=k, \nu_{k} = \nu_{k+1}, \\
e(\nu) \ \ & \text{if} \ m=k+1, \nu_{k}=\nu_{k+1}, \\
0 \ \ & \text{otherwise},
\end{cases} \allowdisplaybreaks\\
& (\tau_{k+1} \tau_{k} \tau_{k+1}-\tau_{k} \tau_{k+1} \tau_{k}) e(\nu)\\
& =\begin{cases} \dfrac{Q_{\nu_{k}, \nu_{k+1}}(x_{k},
x_{k+1}) - Q_{\nu_{k}, \nu_{k+1}}(x_{k+2}, x_{k+1})} {x_{k} -
x_{k+2}}e(\nu) \ \ & \text{if} \
\nu_{k} = \nu_{k+2}, \\
0 \ \ & \text{otherwise}.
\end{cases}
\end{align*}

Note that $R(\beta)$ has the  $\Z$-grading defined by 
\begin{equation*} \label{eq:Z-grading}
\deg e(\nu) =0, \quad \deg\, x_{k} e(\nu) = (\alpha_{\nu_k}
, \alpha_{\nu_k}), \quad\deg\, \tau_{l} e(\nu) = -
(\alpha_{\nu_l} , \alpha_{\nu_{l+1}}).
\end{equation*}

For a $\Z$-graded algebra $A$ over $\mathbf k$, let us denote by $A \Mod$ the category of  graded left modules over $A$ with homogeneous homomorphism. 
Let us denote by $A \proj$ (respectively, $A \gmod$) the full subcategory of $A \Mod$ consisting of
finitely generated projective graded $A$-modules (respectively, finite-dimensional graded $A$-modules). 
Set $R\proj := \bigoplus_{\beta \in \rlQ_+}R(\beta) \proj$ and 
$R\gmod := \bigoplus_{\beta \in \rlQ_+}R(\beta) \gmod$.

For $M\in R(\beta)\gmod$, the $q$-character of $M$ is given by 
$$
\qch(M) = \sum_{\nu \in I^\beta} \dim_q(e(\nu)M) \nu,
$$
where $ \dim_q V := \sum_{n\in \Z} q^n \dim V_n$ for a $\Z$-graded vector space $V = \bigoplus_{k \in \Z} V_k$.
 For an $R(\beta)$-module $M=\bigoplus_{k \in \Z} M_k$, we define
$qM =\bigoplus_{k \in \Z} (qM)_k$, where
 \begin{align*}
 (qM)_k  :=M_{k-1} & \quad \text{for all } \ k \in \Z.
 \end{align*}
We call $q$ the \emph{grading shift functor} on the category of
graded $R(\beta)$-modules.

For $\beta, \gamma \in \rlQ_+$,  we set $e(\beta,\gamma):=\displaystyle\sum_{ \nu_1 \in I^{\beta},  \nu_2 \in I^{\gamma} } e(\nu_1*\nu_2)$, where $\nu_1*\nu_2$ is 
the concatenation of $\nu_1$ and $\nu_2$. The \emph{induction functor} is defined as 
\begin{eqnarray*}
\Ind_{\beta,\gamma} : R(\beta) \otimes R(\gamma) \Mod &\to& R(\beta+\gamma) \Mod \\
V& \mapsto& R(\beta+\gamma)e(\beta,\gamma) \otimes_{R(\beta) \otimes R(\gamma)}  V.
\end{eqnarray*}
We often write 
$$
M\conv N =\Ind_{\beta,\gamma} (M\otimes N),
$$
for $M \in R(\beta)\Mod$ and $N\in R(\gamma)\Mod$, which is called the \emph{convolution product}
of $M$ and $N$.
We denote by $M\hconv N$ and  $M\sconv N$ the \emph{head} and the \emph{socle} of $M \conv N$, respectively.
For simple modules $M, N \in R\gmod$, we say that $M$ and $N$ \emph{strongly commute} if $M\conv N$ is simple. We say that a simple module $L$ is \emph{real} if $L$ strongly commutes with itself.
The \emph{restriction functor} 
$$\Res_{\beta,\gamma} : R(\beta+\gamma) \Mod \to R(\beta) \otimes R(\gamma) \Mod$$
is defined as
 $$
\Res_{\beta,\gamma} (W) : = e(\beta,\gamma)W\simeq e(\beta,\gamma)R(\beta+\gamma) \otimes_{R(\beta+\gamma)} W
  $$
for $W \in R(\beta+\gamma) \Mod$. Note that $\Res_{\beta,\gamma}$ is a  right adjoint of  $\Ind_{\beta,\gamma}$.

Let $\psi$ be the algebra anti-involution on $R(\beta)$ which fixes the generators.
Then $\psi $ gives a left $R$-module structure on the linear dual  $M^\star=\Hom_{\mathbf k}(M,{\mathbf k})$ of a module $M \in R\gmod$.
We call $M$ is  \emph{self-dual} if $M \simeq M^\star$.
 The \emph{Khovanov-Lauda paring} (\cite[(2.43)]{KL09})
$$
(\ , \ ): [R\proj] \times [R\gmod] \longrightarrow \A
$$
 is defined by $$([P], [M]) := \dim_q P^{  \psi } \otimes_R M$$ for $P\in R\proj$ and $M\in R\gmod$.
Here $P^\psi$ is the right $R$-module induced from $P$ with the actions twisted by $\psi$.  It also defines a paring  $[R\proj]\otimes[R\proj] \times [R \gmod]\otimes[R \gmod] \to \A$
by 
\begin{equation}
([X_1]\otimes[X_2], [Y_1]\otimes [Y_2]) := ([X_1],[Y_1])([X_2],[Y_2]) 
\end{equation}
for $ X_1,X_2 \in [R\proj], Y_1,Y_2 \in [R \gmod]$.

\begin{lem} {\rm(\cite[Proposition 3.3]{KL09})} \label{Lem: Ind-Res}
For $P\in R\proj$ and $M,N \in R\gmod$, we have
\[
([P], [M \conv N]) = ([\Res (P)], [M]\otimes[N]).
\]
\end{lem}

The Grothendieck groups $[R\proj]$ and $[R\gmod]$ admit $\A$-algebra structures with the multiplication given by the functor 
$\Ind:=\bigoplus_{\beta,\gamma \in \rlQ_+}\Ind_{\beta,\gamma}$.

\begin{thm} {\rm(\cite{KL09, KL11, R08})}\label{Thm: categorification}
There exists an $\A$-algebra isomorphism 
\begin{eqnarray*}
\gamma : U_\A^-(\g) &\buildrel \sim \over \rightarrow& [R \proj] 
\end{eqnarray*}
We identify $[R\gmod]$ with $[R\proj]^\vee$ as an $\A$-module using the above pairing $( \ , \ )$.
Then
The dual map of $\gamma$
\begin{align*}
\gamma^\vee : [R \gmod] \buildrel \sim \over \rightarrow \Aqq
\end{align*}
is an $\A$-algebra isomorphism.
\end{thm}

Note that, for $u \in U_\A^-(\g)$ and $M \in R\gmod$, 
\begin{align} \label{Eq: ()=<>} 
( \gamma(u), [M]) = \langle u, \gamma^\vee[M] \rangle.
\end{align}
 In the sequel, we 
 identify $U^-_\A(\g)$ (respectively, $\Aqq$) with $[R\proj]$ (respectively, $[R \gmod]$) via the isomorphism $\gamma$ (respectively,  $\gamma^\vee$).

\begin{lem}{\rm(\cite[Proposition 3.2]{KL09})}\label{Lem: Res} The functor $\Res$ categorify $\cmn$. That is, 
under the isomorphism $\gamma$, the map $[\Res]$ on $[R \proj]$ corresponds to the map $\cmn$ on $U_\A^-(\g)$.
\end{lem}

\subsection{Cyclotomic quotients}

For $\Lambda \in \wlP^+$ and  $\beta \in \rlQ_+$, let $I^{\Lambda}_{\beta}$ be the two-sided ideal of $R(\beta)$  generated by the elements
$\{ x^{\langle h_{\nu_{\Ht{\beta}}}, \Lambda \rangle} e(\nu)  \mid \nu \in I^\beta \}$.
The \emph{cyclotomic quiver Hecke algebra} $R^\Lambda(\beta)$ is defined as the quotient algebra 
$$R^\Lambda(\beta):=R(\beta)/I^\lambda_\beta.$$

We define the functors $F_i^{\Lambda}$ and $ E_i^{\Lambda}$ 
by
\begin{align*}
F_i^{\Lambda}M := R^{\Lambda}(\beta+\alpha_i)e(\alpha_i,\beta)\otimes_{R^{\Lambda}(\mu)}M, \quad \text{and} \quad
E_i^{\Lambda}M := e(\alpha_i,\beta-\alpha_i)M,
\end{align*}
for $M\in R^{\Lambda}(\beta)\Mod$. Then the actions of the Chevalley generators are given by the following:
\begin{equation*} 
\begin{aligned}
\xymatrix{
R^\Lambda(\beta)\proj  \ar@<0.3em>[rrrr]^{F_i^{\Lambda}} & & &  &  \ar@<0.3em>[llll]^{q_i^{(1-\langle h_i, \Lambda-\beta \rangle)}E_i^\Lambda}  R^\Lambda(\beta+\alpha_i)\proj   \\R^\Lambda(\beta)\gmod \ar@<0.3em>[rrrr]^{q_i^{(1-\langle h_i, \Lambda-\beta \rangle)}F_i^{\Lambda}} & & &  &  \ar@<0.3em>[llll]^{E_i^\Lambda}  R^\Lambda(\beta+\alpha_i)\gmod 
}
\end{aligned}
\end{equation*}

\begin{thm} {\rm (\cite{KK11})}
There exists an $\A$-module isomorphism 
\begin{eqnarray}
[R^{\Lambda} \proj]  \buildrel \sim \over \longrightarrow V_\A(\Lambda) .
\end{eqnarray} 
Taking the dual,
we have an $\A$-module isomorphism 
$$[R^{\Lambda} \gmod] \buildrel \sim \over \longrightarrow V_\A(\Lambda)^\vee.$$
Under these isomorphisms, we have 
$$\prj_\Lambda =[R^{\Lambda} \otimes_R -] \quad \text{and} \quad \inj_\Lambda = [\mathrm{infl}_\Lambda]$$
where 
$R^{\Lambda} \otimes_R - := \bigoplus_{\beta \in \rlQ_+} (R^\Lambda(\beta) \otimes_{R(\beta)} -) : R \proj \to R^\Lambda \proj $, and
$\mathrm{infl}_\Lambda : R^\Lambda \gmod \to R \gmod $ is the fully faithful functor induced by the surjective ring homomorphism $R(\beta) \twoheadrightarrow R^\Lambda(\beta)$.
\end{thm}

 In the sequel, we omit the map $\inj_\Lambda$ on $V_\A(\Lambda)$ (respectively, on  $B(\Lambda)$) and the functor $\mathrm{infl}_\Lambda$ if no confusion arises.
Note that there exists a pairing
\begin{eqnarray*} 
[R^\Lambda \proj ] \times [R^\Lambda \gmod] &\to& \A \\
(P,M) &\mapsto& \dim_q P^\psi \otimes _{R^\Lambda} M,
\end{eqnarray*}
by which $[R^\Lambda \proj]$ and $[R^{\Lambda}\gmod]$ are dual to each other.

\begin{prop} \cite[Proposition 2.21]{APS16} \label{Prop: conv}
Let $M \in R^\lambda\gmod $ and $N \in R^{\mu}\gmod$. Then the convolution $M \conv N$ is an $R^{\lambda + \mu}$-module.
\end{prop}

\vskip 2em 

\section{Tensor products and convolution products} \label{Sec: tensor and convolution}

\subsection{Tensor products}  \

The following lemma is known, but   
we  include a proof for the convenience of the reader.
\begin{lem} [\protect{cf.  \cite[Lemma 2.5]{Kimura12}}]  \label{Lem: comulti}
For $x \in U_q^-(\g)$ with $\cmn(x) = x_{(1)} \otimes x_{(2)}$, we have 
\[
\cmm(x) = x_{(1)} q^{- \xi(\wt(x_{(2)}))} \otimes x_{(2)},
\]
 where 
 $\xi : \h^* \to \h$ is the  $\Q$-linear isomorphism determined by 
$$\lan \xi(\lambda),\mu \ran =(\lambda,\mu) \quad \text{for all } \ \lambda,\mu \in \mathfrak h^*.$$
\end{lem}
\begin{proof}
We shall use induction on $\Ht(x)$. 
As it is trivial when $\Ht(x)=0$, we assume that the assertion holds for $x$ with $\Ht(x)>0$.

Let $i\in I$. By the definition of $\cmn$, we have  
\begin{align*}
\cmn(f_ix) = \cmn(f_i)\cmn(x) = f_ix_{(1)} \otimes x_{(2)} + q^{(\alpha_i, \wt(x_{(1)}))}   x_{(1)} \otimes  f_i x_{(2)}.
\end{align*}
On the other hand, by the induction hypothesis,
\begin{align*}
\cmm(f_i x) &=  (f_i \otimes 1 + t_i \otimes f_i) (x_{(1)}q^{-\xi(\wt(x_{(2)}  )  )} \otimes x_{(2)} ) \\
&= f_i x_{(1)} q^{-\xi( \wt(x_{(2)}) ) } \otimes x_{(2)} +   t_i x_{(1)} q^{-\xi( \wt(x_{(2)}) )} \otimes f_i x_{(2)} \\
 &= f_i x_{(1)} q^{-\xi( \wt(x_{(2)}) ) } \otimes x_{(2)} +   q^{(\alpha_i, \wt(x_{(1)}) ) }  x_{(1)} t_i q^{-\xi( \wt(x_{(2)}) )} \otimes f_i x_{(2)} \\
 &= f_i x_{(1)} q^{-\xi( \wt(x_{(2)}) ) } \otimes x_{(2)} +   q^{(\alpha_i, \wt(x_{(1)}) ) } x_{(1)}  q^{-\xi(  -\alpha_i + \wt(x_{(2)}) )} \otimes f_i x_{(2)}.
\end{align*}
Therefore, the assertion holds for $f_i x$ which complete the proof.
\end{proof}

\medskip
Let $\lambda \in \wlP^+$. 
We now consider the $\Q(q)$-linear map 
\begin{align*}
\fc: U_q^-(\g) \rightarrowtail U_q^-(\g) \otimes_{\Q(q)} U_q^-(\g) 
\end{align*}
defined by $$\fc(x) = q^{-(\wt(x_{(2)}), \lambda)} x_{(1)} \otimes x_{(2)}, $$ where $\cmn(x) = x_{(1)} \otimes x_{(2)}$.
Note that
\begin{align} \label{Eq: fc}
\fc|_{U_\A^-(\g)} : U_\A^-(\g) \rightarrowtail U_\A^-(\g) \otimes_{\A} U_\A^-(\g). 
\end{align}
We simply write $\fc$ instead of $\fc|_{U_\A^-(\g)}$ if there is no afraid of confusion.

Define
$$(U_q^-(\g) \otimes_{\Q(q)} U_q^-(\g))^\vee_\A 
:=\{ f \in (U_q^-(\g) \otimes_{\Q(q)} U_q^-(\g))^\vee \mid  f(U_\A^-(\g) \otimes_{\A} U_\A^-(\g)) \subset \A \}.
$$
By restricting the $\Q(q)$-linear isomorphism
$$  
\Aq \otimes_{\Q(q)} \Aq \simeq  (U_q^-(\g) \otimes_{\Q(q)} U_q^-(\g))^\vee
$$
which is given by
$
\lan f\otimes g , x\otimes y \ran := f(x)g(x) 
$
 for $ f, g\in \Aq $\ and $ x,y \in U_q^-(\g)$
on the $\A$-lattice 
$\Aqq \otimes_{\A} \Aqq \subset \Aq \otimes_{\Q(q)} \Aq,$
we obtain an $\A$-module  isomorphism 
\begin{equation} \label{eq: eta}
 \Aqq \otimes_{\A} \Aqq \buildrel \sim\over \longrightarrow  (U_q^-(\g) \otimes_{\Q(q)} U_q^-(\g))^\vee_\A.
\end{equation}
Similarly, by restricting \eqref{eq:MtensNdual} on the $\A$-lattice 
$V_\A(\lambda)^\vee \otimes_{\A} V_\A(\mu)^\vee \subset V_q(\lambda)^\vee \otimes_+ V_q(\mu)^\vee$,
we obtain an 
$\A$-module  isomorphism 
\begin{equation} \label{eq: zeta}
  V_\A(\lambda)^\vee \otimes_{\A} V_\A(\mu)^\vee  \buildrel \sim\over \longrightarrow (V_q(\lambda) \otimes_- V_q(\mu))^\vee_\A ,
\end{equation} where 
$(V_q(\lambda) \otimes_- V_q(\mu))^\vee_\A 
:=\{f \in (V_q(\lambda) \otimes_- V_q(\mu))^\vee \mid f(V_\A(\lambda) \otimes_\A V_\A(\mu)) \subset \A \}.$
In the sequel, we will identify the spaces  in \eqref{eq: eta} and \eqref{eq: zeta} respectively,  without mentioning the isomorphisms.

Now, restricting the dual map $ (U_q^-(\g) \otimes_{\Q(q)} U_q^-(\g))^\vee \buildrel (\fc)^\vee \over \twoheadrightarrow U_q^-(\g)^\vee  $ to $(U_q^-(\g) \otimes_{\Q(q)} U_q^-(\g))^\vee_\A \simeq \Aqq \otimes_\A \Aqq$, we have
\begin{align} \label{Eq: dfc}
\dfc :=  (\fc )^\vee 
: \Aqq \otimes_\A \Aqq \twoheadrightarrow \Aqq
\end{align}

Let $\lambda,\mu \in \wlP^+$ .
We define a $U_q(\g)$-module homomorphism 
\begin{align*}
 \iota_{\lambda, \mu} : &  V_q(\lambda + \mu) \rightarrowtail V_q(\lambda ) \otimes_{-}V_q(\mu)
\end{align*}
by $\iota_{\lambda, \mu}(u_{\lambda+ \mu}) = u_\lambda \otimes u_\mu$. By Lemma \ref{Lem: comulti} and $\eqref{Eq: fc}$, 
for $x \in U_\A^-(\g) $, we have 
\begin{align} \label{Eq: compatible with f_i}
x \cdot (u_\lambda \otimes u_\mu) = \cmm(x) (u_\lambda \otimes u_\mu) = \fc(x) (u_\lambda \otimes u_\mu) \in V_\A(\lambda) \otimes_\A V_\A(\mu).
\end{align}
Restricting $\iota_{\lambda, \mu}$ to $V_\A(\lambda + \mu)$, we have 
\begin{align} \label{Eq: iota_A}
 \iota_{\lambda, \mu} |_{ V_\A(\lambda + \mu)} : &  V_\A(\lambda + \mu) \rightarrowtail V_\A(\lambda ) \otimes_{ \A }V_\A(\mu)  \subset V_q(\lambda) \otimes_- V_q(\mu). 
\end{align}
We simply write $ \iota_{\lambda, \mu}$ instead of $ \iota_{\lambda, \mu} |_{ V_\A(\lambda + \mu)}$ if there is no afraid of confusion.
Similarly, restricting the dual map $  (V_q(\lambda ) \otimes_{-}V_q(\mu))^{\vee}  \buildrel (\iota_{\lambda, \mu})^{\vee} \over \twoheadrightarrow  V_q(\lambda + \mu)^\vee  $
to  $ (V_q(\lambda ) \otimes_{-}V_q(\mu))^\vee_\A \simeq V_\A(\lambda )^\vee  \otimes_{\A}V_\A(\mu)^\vee  $,  we have 
\begin{align} \label{Eq: pi_A}
\pi_{\lambda, \mu}:=  \iota_{\lambda, \mu}^\vee 
: V_\A(\lambda )^\vee \otimes_{ \A }V_\A(\mu)^\vee \twoheadrightarrow V_\A(\lambda + \mu )^\vee.
\end{align}

We now consider the following commutative diagram:
\begin{align} \label{Eq: comm}
\xymatrix
{
 U_q^-(\g) \ar@{>->}[rr]^{\fc} \ar@{->>}[d]_{\prj_{\lambda + \mu}} &&  U_q^-(\g) \otimes U_q^-(\g)    \ar@{->>}[rr]^{\prj_\lambda \otimes \prj_\mu } &&  V_q(\lambda) \otimes_{-} V_q(\mu) \\
 V_q(\lambda + \mu) \ar@{>->}[rrrru]_{\iota_{\lambda, \mu}}
}
\end{align} 
Note that, since $\iota_{\lambda,\mu}$ is injective, so is $\fc$. 
Restricting $\eqref{Eq: comm}$ to the $\A$-lattices, by  $\eqref{Eq: proj_A}$, $\eqref{Eq: fc}$ and $\eqref{Eq: iota_A}$, we have the following:
\begin{align} \label{Eq: comm_proj}
\xymatrix
{
 U_\A^-(\g) \ar@{>->}[rr]^{\fc} \ar@{->>}[d]_{\prj_{\lambda + \mu}} &&  U_\A^-(\g) \otimes U_\A^-(\g)    \ar@{->>}[rr]^{\prj_\lambda \otimes \prj_\mu } &&  V_\A(\lambda) \otimes_{ \A } V_\A(\mu) \\
 V_\A(\lambda + \mu) \ar@{>->}[rrrru]_{\iota_{\lambda, \mu}}
}
\end{align} 
Taking the dual of $\eqref{Eq: comm}$ and restricting it to the $\A$-lattices, by $\eqref{Eq: Inj_A}$, $\eqref{Eq: dfc}$ and $\eqref{Eq: pi_A}$, we have the following  commutative diagram:
\begin{align} \label{Eq: comm_inj}
\xymatrix
{
 \Aqq   && \ar@{->>}[ll]_{\dfc}  \Aqq \otimes_{ \A } \Aqq     && \ar@{>->}[ll]_{\inj_\lambda \otimes \inj_\mu }   V_\A(\lambda)^\vee \otimes_{ \A} V_\A(\mu)^\vee \ar@{->>}[lllld]^{\pi_{\lambda, \mu}}. \\
 V_\A(\lambda + \mu)^\vee  \ar@{>->}[u]^{\inj_{\lambda + \mu}}  
}
\end{align}

\subsection{Convolution products} \

For $\lambda \in \wlP^+$  and $M \in R(\beta)\Mod$, we define 
\begin{equation}
\begin{aligned}
\mRes(M) & := \bigoplus_{ \beta_1, \beta_2 \in \rlQ_+,  \  \beta_1 + \beta_2=\beta} q^{(\beta_2, \lambda) } \Res_{\beta_1, \beta_2}(M). 
\end{aligned}
\end{equation}
Thanks to \cite[Proposition 2.19]{KL09}, $\mRes$ takes projective modules to projective modules. Thus we have the functors 
$$
\mRes : R\proj \rightarrow (R\otimes R) \proj,
$$ 
which give the $\A$-linear maps at the level of Grothendieck groups
$$
[\mRes] : [R\proj] \rightarrow [(R\otimes R) \proj] \simeq [R\proj]\otimes[R\proj].
$$

\begin{lem} \label{Lem: Res_lam}
The functors $\mRes$ categorify $\fc$. That is, for any $P \in R \proj$ we have
$$[\mRes P]= \fc ([P]).$$
\end{lem}
\begin{proof}
It follows from Lemma \ref{Lem: Res} and Lemma \ref{Lem: comulti}.
\end{proof}

\begin{thm} \label{Thm: main}
For $M \in R(\beta_1)\gmod$ and $N\in R(\beta_2)\gmod$, we have
\[
\dfc( [M] \otimes [N] ) = q^{(\beta_2, \lambda)}[M\conv N].
\]
\end{thm}
\begin{proof}
By the definition, we have $\langle u, \dfc(m\otimes n)  \rangle =  \langle \fc(u), m\otimes n  \rangle  $
for any $u \in U_\A^-(\g)$ and $m,n \in U_\A^-(\g)^\vee$.
As we identify $U^-_\A(\g) \simeq [R\proj]$ and $[R \gmod] \simeq  \Aqq$ via $\gamma$  and $\gamma^\vee$ given in Theorem \ref{Thm: categorification} respectively, 
it follows from $\eqref{Eq: ()=<>}$ that 
$$
( [P], \dfc([M] \otimes [N]))  =   (\fc([P]), [M] \otimes [N])
$$
for any $P \in R\proj$ and $M,N \in R\gmod$.
By Lemma \ref{Lem: Ind-Res} and Lemma \ref{Lem: Res_lam}, for any $P \in R\proj$, $M \in R(\beta_1)\gmod$ and $N \in R(\beta_2)\gmod$, we have 
\begin{align*}
( [P], \dfc([M] \otimes [N])) & =   (\fc([P]), [M] \otimes [N]) = ([\mRes(P)] , [M] \otimes [N]) \\ 
&= ( q^{(\beta_2, \lambda)} [\Res_{\beta_1, \beta_2}(P)] , [M] \otimes [N]) \\ 
&= (  [P] , q^{(\beta_2, \lambda)} [M \conv N]), 
\end{align*} 
which implies the assertion. 
\end{proof}

Then we have the following corollary.
\begin{cor} \label{Cor: main}
For $M \in R^\lambda(\beta_1)\gmod$ and $N\in R^\mu(\beta_2)\gmod$, we have
\[
 \pi_{\lambda, \mu} ( [M] \otimes [N] ) = q^{(\beta_2, \lambda)}[M\conv N] 
\] 
\end{cor}
\begin{proof}
The assertion follows 
from $\eqref{Eq: comm_inj}$, Proposition \ref{Prop: conv} and Theorem \ref{Thm: main}. 
\end{proof}

\subsection{Symmetric case}
From now on, we assume that the base field $\bR$ is of characteristic 0 and $R$ is \emph{symmetric}; that is,
the parameter
$$ \text{$Q_{i,j}(u,v)$ is a polynomial in $u-v$ for any
$i,j \in I$.}$$ 
Then it was proved in \cite{R11, VV11} that the set of self-dual simple $R$-modules in $R\gmod$ corresponds the upper global basis of $\Aqq$ under the isomorphsim $\gamma^\vee$. 
Combining this with the connection between tensor products and convolution products,
we have interesting theorems below. 
For $b \in B(\infty)$, we denote by $L(b)$ the self-dual simple R-module corresponding to $G^\up(b)$. 
For $b \in B(\lambda)$, we simply use the same notation $L(b)$ for the self-dual simple $R^\lambda$-module corresponding to $G^\up_\lambda(b)$ if there is no afraid of confusion.

\begin{thm} \label{Thm: positivity}
Let $M \in R^\lambda(\beta_1)\gmod$ and $N\in R^\mu(\beta_2)\gmod$. 
Suppose that  
$$
  \pi_{\lambda, \mu} ( [M] \otimes [N] ) = \sum_{b \in B(\lambda+\mu)} A_b(q)    [L(b)]  
$$ 
for $A_b(q) \in  \Q(q) $. Then, for $b\in B(\lambda+\mu)$, we have 
\begin{enumerate}
\item $A_b(q) \in \Z_{\ge0}[q,q^{-1}]$,
\item $[M\conv N : L(b)]_q = q^{-(\beta_2, \lambda)} A_b(q).$
\end{enumerate}  
\end{thm}
\begin{proof}
As the set of simple $R$-modules corresponds the upper global basis of $\Aqq$, the assertion follows from Corollary \ref{Cor: main}.  
\end{proof}

We have the following corollary  immediately. 

\begin{cor} \label{Cor: positivity}
For $b \in B(\lambda)$ and $b' \in B(\mu)$,  we write   
$$
  \pi_{\lambda, \mu} ( G^{\up}_\lambda(b) \otimes G^{\up}_\mu(b') ) = \sum_{b'' \in B(\lambda+\mu)} A_{b''}(q)   G_{ \lambda+\mu  }^{\up}(b'') 
$$ 
for $A_{b''}(q) \in  \Q(q) $. Then, we prove that 
\begin{enumerate}
\item $A_{b''}(q) \in \Z_{\ge0}[q,q^{-1}]$,
\item $[L(b)\conv L(b') : L(b'')]_q = q^{-(\beta_2, \lambda)} A_{b''}(q)$ for $b''\in B(\lambda+\mu)$. 
\end{enumerate}  
\end{cor}

Let us recall that $C_{\lambda_{1}, \ldots, \lambda_r }$ denotes the connected component of the crystal $B(\lambda_1)\otimes \cdots \otimes B(\lambda_r)$ 
containing the highest weight vector $u_{\lambda_1}\otimes \cdots \otimes u_{\lambda_r}$ for $\lambda_1, \ldots, \lambda_r \in \wlP^+$.

\begin{thm} \label{Thm: app_tensor}
Let $\lambda_k \in \wlP^+$ and $b_k \in B(\lambda_k)$ for $k=1, \ldots, r$.
We set $\beta_k := \lambda_k - \wt(b_k)$.
If 
\begin{enumerate}
\item[(i)] $b_1 \otimes \cdots \otimes b_r \in C_{\lambda_{1}, \ldots, \lambda_r }$, 
\item[(ii)] $\hd( L(b_1) \conv \cdots \conv L(b_r) )$ is simple,
\end{enumerate}
then we have 
\[
 q^t   \, \hd( L(b_1) \conv \cdots \conv L(b_r) ) \simeq L( b_1 \otimes \cdots \otimes b_r )
\]
 where $t={\displaystyle\sum_{1 \le i < j \le r} (\beta_j,\lambda_i)}$.
\end{thm}
\begin{proof}
It was shown in \cite[Theorem 4.2.1]{KKKO17}  that 
for $M \in R\gmod$,
if $\hd (M)$ is simple and  $q^{-s} (\hd (M))$ is self-dual for some $s\in \Z$, then in the Grothendieck group $[R\gmod]$
\[
[M] = [\hd (M)] + \sum_{k} q^{s_k}[S_k] 
\]
for some self-dual simple modules $S_k$ and some $s_k > s$.

We choose an integer $c$  such that $q^{ c  }\hd (L(b_1) \conv L(b_2)\conv \cdots \conv L(b_r))$ is self-dual.
Then we have 
\[
q^{ c} [L(b_1) \conv L(b_2)\conv \cdots \conv L(b_r) ] = q^{c} [\hd (L(b_1) \conv L(b_2)\conv \cdots \conv L(b_r)) ] + \sum_{k} q^{c_k}[S_k] 
\]
for some self-dual simple modules $S_k$ and some $c_k > 0$ in $[R^\lambda\gmod]$, where $\lambda=\lambda_1+\cdots +\lambda_r$. 
In other word, we know that 
$q^{  c } [\hd (L(b_1) \conv L(b_2)\conv \cdots \conv L(b_r)) ]$ belongs to the upper global basis and 
\[
q^{  c } [L(b_1) \conv L(b_2)\conv \cdots \conv L(b_r) ] \equiv q^{ c} [\hd (L(b_1) \conv L(b_2)\conv \cdots \conv L(b_r)) ] \mod q L^\up( \lambda ),
\]
since the upper global basis is an $\A_0$-basis of $L^\up(\lambda)$.

On the other hand, by Corollary \ref{Cor: main} and Proposition \ref{Prop: global}
we have 
\begin{align*}
q^{t}  [L(b_1) \conv \cdots \conv L(b_r) ]  &=  \pi_{\lambda_1, \ldots, \lambda_r} ([L(b_1)] \otimes [L(b_2)] \otimes \cdots \otimes [L(b_r)]) \\ 
&\equiv [L(b_1 \otimes b_2 \otimes \cdots \otimes b_r)] \qquad {\mathrm{mod} \ \ qL^{\up} (  \lambda ). }
\end{align*}
 Hence we conclude that  $c=t$ and 
\[
  q^t  \hd( L(b_1) \conv \cdots \conv L(b_r) ) \simeq L( b_1 \otimes \cdots \otimes b_r ),
\]
as desired.
\end{proof}

Recall that for each pair of nonzero modules $M$ and $N$ in $R\gmod$,  there exists an integer $\Lambda(M,N)$ and a distinguished nonzero homomorphism $\rmat{M,N} : M\conv N \to q^{-\Lambda(M,N)}N \conv M$, called the \emph{R-matrix}.  For the definition and properties of R-matrices, we refer \cite[Section 2.2, 3.1, 3.2]{KKKO17}.
Set 
$$\widetilde \Lambda(M,N) : = \dfrac{1}{2}(\Lambda(M,N)+(\wt M ,\wt N)), \quad 
\Dd(M,N) := \dfrac{1}{2}(\Lambda(M,N) + \Lambda(N,M)).
$$
If $M$ and $N$ are simple $R$-modules and one of them is real, then 
 $M$ and $N$ strongly commute if and only if $\Dd(M,N)=0$ (\cite[Lemma 3.2.3]{KKKO17}).

\begin{cor} \label{Cor: app_tensor}
Let $\lambda_1, \lambda_2 \in \wlP^+$ and $ b_1 \otimes b_2 \in C_{\lambda_1, \lambda_2}$. 
We set $\beta_{k} := \lambda_k -\wt  (b_k)  \in  \rlQ_+$ for $k=1,2$.
Suppose that  one of $L(b_1)$ and $L(b_2)$ is real. Then we have
\begin{enumerate}
\item  $ q^{(\beta_2, \lambda_1)}   L(b_1) \hconv L(b_2) \simeq L(b_1 \otimes b_2)$,
\item 
$\widetilde \Lambda(L(b_1), L(b_2)) =(\beta_2, \lambda_1)$ and  
$\Lambda(L(b_1),L(b_2))=(\beta_2, 2\lambda_1-\beta_1)$,
\item 
if there are $\lambda_1', \lambda_2' \in \wlP^+$ and  $\ b'_2\otimes b'_1 \in C_{\lambda'_2,\lambda'_1}$ such that 
$L(b_1)\simeq L(b_1')$ and $L(b_2) \simeq L(b_2')$,
then we have
$$\Dd(L(b_1), L(b_2))=(\beta_1,\lambda'_2)+(\beta_2,\lambda_1)-(\beta_1,\beta_2)=(\lambda_1,\lambda'_2)-(\wt(b_1), \wt (b'_2)).$$
In particular, $L(b_1)$ and $L(b_2)$ strongly commute if and only if $(\lambda_1,\lambda_2')=(\wt (b_1), \wt (b_2')).$
\end{enumerate}
\end{cor}
\begin{proof}
(1) As either of $L(b_1)$ and $L(b_2)$ is real, $L(b_1)\conv L(b_2)$ has a simple head by \cite[Theorem 3.2]{KKKO15a}.
Thus it follows  from  Theorem \ref{Thm: app_tensor}.

(2)
 By (1), we know that $q^{(\beta_2,\lambda_1)}L(b_1)\hconv L(b_2)$ is self-dual. It follows that 
$(\beta_2,\lambda_1) = \widetilde \Lambda(L(b_1), L(b_2))$ by \cite[Lemma 3.1.4]{KKKO17}. 
The second follows from it by the definition.

(3)
It is straightforward to prove from (1) and (2).
\end{proof}

\begin{lem} \label{Lem: nonzero}
Let $V$ and $U$ be nonzero  
modules over a  ring $A$, and 
$f : V \rightarrow U$ be a nonzero $A$-module homomorphism. Suppose that  
$V$ has a unique maximal submodule $M$. Then the induced homomorphism 
$$
\overline{f}: V / M \rightarrow U/ f(M) \text{ is nonzero.}
$$
\end{lem}
\begin{proof}
Without loss of generality, we may assume that $f$ is surjective.
Then, it suffices to show that $f(M) \ne U$.

We assume that $f(M) = U$. Let $x \in V \setminus M$ and take $y \in M$ such that $f(x) = f(y)$. We set $v : = x-y \in V$. Then we have
$v \notin M$, $f(v)=0$ and $A v = V$ since $A v \not\subset M$. 
Thus  $\Im(f) = f(V) = f(Av) = 0$, which is a contradiction to nonzeroness of $f$.
\end{proof}

\begin{thm} \label{Thm: sR}
Let $\Lambda, \lambda_1, \lambda_2, \lambda_1', \lambda_2' \in \wlP^+$ with 
$ \Lambda = \lambda_1+\lambda_2 = \lambda_1' + \lambda_2' $, and
let  $b_k \in B(\lambda_k)$, $b'_k \in B(\lambda_k')$ for $k=1,2$.
We set $\beta_k := \lambda_k - \wt(b_k)$ and $\beta_k' := \lambda_k' - \wt(b_k')$ for $k=1,2$.
Suppose that  
\begin{enumerate}
\item[(i)] $b_1 \otimes b_2 \in C_{\lambda_1, \lambda_2}$,
\item[(ii)] $b_1 \otimes b_2 \simeq b'_1 \otimes b'_2$,
\item[(iii)] one of $G^{up}(b_1)$ and $G^{up}(b_2)$ (resp.\  one of  $G^{up}(b'_1)$ and $G^{up}(b'_2)$ ) is real. 
\end{enumerate}
Then we have 
\begin{enumerate}
\item $ \Hom_{R\gmod} ( q^d L(b_1) \conv  L(b_2),  L(b'_2) \conv L(b'_1))  \simeq {\mathbf k}$,  where $d={(\beta_2,\lambda_1) +(\beta'_2,\lambda_1') - (\beta_1', \beta_2')}$, 
\item  for any nonzero homomorphism $\varphi \in  \Hom_{R\gmod} ( q^d L(b_1) \conv  L(b_2),  L(b'_2) \conv L(b'_1)) $, we have  
$$
q^d \Im (\varphi) \simeq q^d L(b_1) \hconv L(b_2) \simeq  L(b'_2) \sconv L(b'_1). 
$$

\end{enumerate}
\end{thm}
\begin{proof}
As $b_1 \otimes b_2 \in C_{\lambda_1, \lambda_2}$ and $b_1 \otimes b_2 \simeq b'_1 \otimes b'_2$, the connected component of $B(\lambda_1')\otimes B(\lambda_2')$ containing $b'_1\otimes b'_2$
should be isomorphic to $B(\Lambda)$. Thus, $b'_1 \otimes b'_2$ is contained in $C_{\lambda_1', \lambda_2'}$ since 
$C_{\lambda_1', \lambda_2'}$ is a unique connected component of $B(\lambda_1')\otimes B(\lambda_2')$ which is isomorphic to $B(\Lambda)$.
It follows from Corollary \ref{Cor: app_tensor} that 
\begin{align*} 
  q^{(\beta_2 ,\lambda_1)}  L(b_1) \hconv L(b_2) \simeq L(b_1 \otimes b_2) \simeq L(b'_1 \otimes b'_2) \simeq 
  q^{(\beta'_2 ,\lambda_1')}  L(b'_1) \hconv L (b'_2).
\end{align*}

 As $L(b'_1) \hconv L(b'_2) \simeq   q^{-\Lambda(L(b'_1),L(b'_2))}  L(b'_2) \sconv L(b'_1)$  
 \cite[Theorem 3.2]{KKKO15a},
we have a chain of homomorphisms 
\begin{align}  \label{Eq: homo}
q^d L(b_1) \conv L(b_2) \twoheadrightarrow q^d L(b_1)\hconv L(b_2) \buildrel \sim \over \rightarrow 
 L(b'_2) \sconv L(b'_1) \rightarrowtail   L(b'_2) \conv L(b'_1),
\end{align}
 where 
 $$
 d={(\beta_2,\lambda_1) - (\beta'_2,\lambda_1')+\Lambda(L(b'_1),L(b'_2))}
  = (\beta_2,\lambda_1) + (\beta'_2,\lambda_1')-(\beta_1', \beta_2')
 $$
by Corollary  \ref{Cor: app_tensor}.
Moreover the composition of the homomorphisms in $\eqref{Eq: homo}$ is nonzero.
Thus  we conclude that 
\[
\dim_{\mathbf k} (  \Hom_{R\gmod} ( q^d L(b_1) \conv  L(b_2),  L(b'_2) \conv L(b'_1))) \ge 1.
\]

We now set $L := q^d L(b_1) \hconv L(b_2) \simeq  L(b'_2)\sconv L(b'_1)$. By \cite[Theorem 4.1.1]{KKKO17}, we know that 
\begin{align} \label{Eq: once}
\text{$L$ appears only once in the composition series of $L(b'_2)\conv L(b'_1)$.}
\end{align}
As $L(b_1)\hconv L(b_2)$ is simple, $L(b_1)\conv L(b_2)$ 
has a unique maximal submodule which is denoted by $M$.
Let $f : L(b_1) \conv L(b_2) \to  L(b'_2) \conv L(b'_1)$ be a nonzero graded homomorphism. 
Then $\eqref{Eq: once}$ tells that the degree of $f$ should be $d$.
We consider the commutative diagram:
\begin{align*} 
\xymatrix
{
q^d L(b_1)\conv L(b_2)  \ar@{->}[d] \ar@{->}[rr]^{f}  &&   L(b'_2)\conv L(b'_1)    \ar@{->}[d]  \\
 q^d (L(b_1)\conv L(b_2) )/M \ar@{->}[rr]^{\overline{f} }   && (L(b'_2)\conv L(b'_1))/f(M)
}
\end{align*} 
where $\overline{f}$ is the homomorphism induced from $f$. Note that $\overline{f}$ is nonzero by Lemma \ref{Lem: nonzero}.
Since $ L \simeq (L(b_1)\conv L(b_2) )/M $ is simple, 
\begin{align*} 
\text{$L$ appears in the composition series of $  (L(b'_2)\conv L(b'_1))/f(M)$.}
\end{align*}

Suppose that $f(M) \ne 0$. 
Then $L$ should appear in the composition series of $f(M)$ since $ L \simeq L(b'_2)\sconv L(b'_1)$. This tells that $L$ appears at least twice in the 
composition series of $  L(b'_2)\conv L(b'_1)$, which is a contradiction to $\eqref{Eq: once}$.
Hence, we conclude that $f(M) =0$, which means that $f$ can be obtained from the composition of the homomorphisms in $\eqref{Eq: homo}$.
Therefore, we complete the proof.
\end{proof}

\begin{cor} \label{Cor: app2}
Let $\lambda_1, \lambda_2 \in \wlP^+$, $b_1 \otimes b_2 \in C_{\lambda_1, \lambda_2}$ and $b_2 \otimes b_1 \in C_{\lambda_2,\lambda_1}$.
Suppose that 
 one of $L(b_1)$ and $L(b_2)$ is real.
Then  the followings are equivalent. 
\begin{enumerate}
\item $b_1 \otimes b_2 \simeq b_2 \otimes b_1$.
\item $\text{$L(b_1)$ and $L(b_2)$ strongly commute}.$
\item $(\lambda,\mu)=(\wt (b_1), \wt (b_2)).$
\end{enumerate}
\end{cor}
\begin{proof}
Since (2) and (3) are equivalent by Corollary \ref{Cor: app_tensor} (2), it suffices to show that (1) and (2) are equivalent.

If (1) holds, then it follows from Theorem \ref{Thm: sR} that $L(b_1) \hconv L(b_2) \simeq   L(b_1) \sconv L(b_2) $ up to a grading shift,
 which implies (2) by \cite[Corollary 3.9]{KKKO15a}.
 
 Suppose that (2) holds.
 By Corollary \ref{Cor: app_tensor} (1), we have 
$$L(b_1\otimes b_2) \simeq L(b_1) \hconv L(b_2) =L(b_1)\conv L(b_2) \simeq L(b_2)\conv L(b_1)  =
L(b_2)\hconv L(b_1) \simeq L(b_2\otimes b_1)
$$
up to grading shifts.
Hence we have 
$G^\up_{\lambda_1+\lambda_2}(\overline{\pi}_{\lambda,\mu}(b_1\otimes b_2)) = G^\up_{\lambda_1+\lambda_2}((\overline{\pi}_{\mu,\lambda}(b_2\otimes b_1))$
and hence  $b_1\otimes b_2 \simeq b_2 \otimes b_1,$ as desired.
\end{proof}

\begin{Rmk}
 In the case of Corollary \ref{Cor: app2}, we actually  obtain that 
\begin{align*} 
d&= (\beta_1,\lambda_2)+(\beta_2,\lambda_1) - (\beta_1, \beta_2)
=\Dd(L(b_1),L(b_2))=0,
\end{align*}
where $d$ is the integer given in Theorem \ref{Thm: sR}.
\end{Rmk}

\vskip 2em

\section{Applications to finite type $A$} \label{Sec: application}

Let $\lambda = (\lambda_1 \ge \lambda_2 \ge \ldots \ge \lambda_l > 0)$ be a {\it Young diagram} of size $|\lambda| = \sum_{k=1}^l \lambda_k$.
We denote by $\lambda'$ the {\it conjugate} of $\lambda$, and write  $\emptyset$ for the Young diagram of $0$.  
We will depict a Young diagram using English convention.
Let $\Par_n := \{ \lambda \mid \ell(\lambda) \le n \} $, where $\ell(\lambda)$ is the number of rows of $\lambda$.
A {\it semistandard tableau} $T$ of shape $\lambda \in \Par_n$ is a filling of boxes of $\lambda$ with entries in $\{ 1, 2,\ldots, n\}$ such that 
(i) the entries in rows are weakly increasing from left to right, (ii) the entries in columns are strictly increasing from top to bottom. 
We write $ \sh(T) = \lambda$. 
Let $\SST(\lambda)$ be the set of all semistandard tableaux of shape $\lambda$.
Similarly, 
a {\it standard tableau} of shape $\lambda$ is a filling of boxes of $\lambda$ with $\{1, \ldots, |\lambda| \}$ such that 
(i) each number is used exactly once, (ii) the entries in rows and columns increase from left to right and from top to bottom, respectively.
Let us denote by $\ST(\lambda)$ the set of all standard tableaux of shape $\lambda$.
We write $b = (p,q) \in T$ if $b$ is a box of $T$ at the $p$-th row and the $q$-th column, and  set $T(b)$ to be the entry in the box $b$.  
For example, if $n=8$ and $\lambda = (10,8,4,2)$, then $T (1,3) = 2$, $\lambda' = (4,4,3,3,2,2,2,2,1,1)$ and the following is a semistandard tableau of shape $\lambda$.
$$
\xy
(0,12)*{};(60,12)*{} **\dir{-};
(0,6)*{};(60,6)*{} **\dir{-};
(0,0)*{};(48,0)*{} **\dir{-};
(0,-6)*{};(24,-6)*{} **\dir{-};
(0,-12)*{};(12,-12)*{} **\dir{-};
(0,12)*{};(0,-12)*{} **\dir{-};
(6,12)*{};(6,-12)*{} **\dir{-};
(12,12)*{};(12,-12)*{} **\dir{-};
(18,12)*{};(18,-6)*{} **\dir{-};
(24,12)*{};(24,-6)*{} **\dir{-};
(30,12)*{};(30,0)*{} **\dir{-};
(36,12)*{};(36,0)*{} **\dir{-};
(42,12)*{};(42,0)*{} **\dir{-};
(48,12)*{};(48,0)*{} **\dir{-};
(54,12)*{};(54,6)*{} **\dir{-};
(60,12)*{};(60,6)*{} **\dir{-};
(3,9)*{1}; (9,9)*{1}; (15,9)*{2}; (21,9)*{3}; (27,9)*{3}; (33,9)*{3}; (39,9)*{5};(45,9)*{6};(51,9)*{7}; (57,9)*{8}; 
(3,3)*{2}; (9,3)*{3}; (15,3)*{3}; (21,3)*{4}; (27,3)*{4}; (33,3)*{6}; (39,3)*{6}; (45,3)*{7}; 
(3,-3)*{4}; (9,-3)*{5}; (15,-3)*{5}; (21,-3)*{7};
(3,-9)*{6}; (9,-9)*{8};
\endxy
$$
\vskip 0.5em

We assume that the Cartan matrix $\cmA$ is of type $A_{n-1}$ with $I=\{1,2, \ldots, n-1 \}$, the field $\bR$ is of characteristic 0 and 
the quiver Hecke algebra $R$ is symmetric. Note that, in type $A_{n-1}$, any quiver Hecke algebra is isomorphic to a symmetric quiver Hecke algebra. 
 Let $\lambda $ be a Young diagram with $\ell(\lambda)<n$  and $\mu := \lambda '$.
We write $\mu  = (\mu_1, \ldots, \mu_r)$  and  set $\Lambda := \sum_{k=1}^r  \Lambda_{\mu_k}$.
Then it was known that $\SST(\lambda)$ has  a $U_q(\g)$-crystal structure and 
\begin{align*} 
\SST( \lambda) \simeq B(\Lambda)
\end{align*}
 as crystals (see \cite{KN94}, \cite[Chapter 7]{HK02} for details). Moreover, considering the tensor product of crystals, we have the crystal embedding 
\begin{align*}
\SST(\lambda) \longrightarrow \SST(\mu_r') \otimes  \SST(\mu_{r-1}') \otimes  \cdots \otimes  \SST(\mu_1') 
\end{align*}
sending $T$ to $ T_r \otimes T_{r-1} \otimes \cdots \otimes T_1$, 
where $T_k$ is the $k$-th column of $T$ from the left.
For $T \in \SST( \lambda)$, we denote by $L(T)$ the self-dual simple $R$-module corresponding to $T$ under the categorification. 

For $k, l \in I$, there is an explicit combinatorial description of a crystal isomorphism 
\begin{align} \label{Eq: cRmatrix}
\sigma: \SST( k') \otimes \SST( l' ) \buildrel \sim \over \longrightarrow  \SST( l' ) \otimes \SST( k' )
\end{align}
which was explained in \cite[Section 3.5]{NY97}.  Note that $\sigma$ comes from the {\it combinatorial $R$-matrix} 
if we regard $\SST( k')$ and $\SST(l')$ as realizations of the Kirillov-Reshetikhin crystals $B^{k,1}$ and $B^{l,1}$ of type $A_{n-1}^{(1)}$. 
Let us explain the crystal isomorphism $\sigma$ briefly. Let $\mathsf{B} = \{ 00, 10, 01, 11 \}$ and define a $U_q(A_1)$-crystal structure by
\begin{align*} 
\te (00) = \te(11) =  \te (10) = \emptyset, \quad  \te(01)= 10 , \\
\tf (00) = \tf(11) =  \tf (01) = \emptyset, \quad  \tf(10)= 01.
\end{align*}
Note that $\mathsf{B} \simeq B(1) \oplus B(0)^{\oplus 2}$.
We identify $\bigsqcup_{0 \le k_1, k_2 \le n } \SST(k_1') \otimes \SST(k_2')$ with $\mathsf{B}^{\otimes n}$ as follows. 
For $T \in \SST(k')$ and $a =1, \ldots, n$, we write $a \in T$ when $a$ appears in $T$ as an entry. 
Then we have the following bijection 
\begin{align} \label{Eq: b}
\mathsf{b}: \bigsqcup_{0 \le k_1, k_2 \le n } \SST(k_1') \otimes \SST(k_2') \buildrel \sim \over \longrightarrow \mathsf{B}^{\otimes n}
\end{align}
defined by $\mathsf{b}(T_1 \otimes T_2) = b_n \otimes b_{n-1} \otimes \cdots \otimes b_1$, where
$$
b_a = 
\left\{
  \begin{array}{ll}
  00 & \text{ if $a \notin T_1$ and $a \notin T_2$,}    \\
  10 & \text{ if $a \in T_1$ and $a \notin T_2$,}    \\
  01 & \text{ if $a \notin T_1$ and $a \in T_2$,}    \\
  11 &  \text{ if $a \in T_1$ and $a \in T_2$.}
  \end{array}
\right.
$$
Then, for $k \le l$, the crystal isomorphism $\sigma$ in $\eqref{Eq: cRmatrix}$ can be described as 
$$
\sigma(T_1 \otimes T_2) = \mathsf{b}^{-1} ( \te^{l - k}  \mathsf{b} ( T_1 \otimes T_2 )  ).
$$

Let $k\in I$. Every vector of the crystal $B(\Lambda_k)$ is extremal, which tells that $R^{\Lambda_k}(\beta)$ is a simple algebra 
for any $\beta \in  \Lambda_k - \wt( B(\Lambda_k) )$. 
For $T \in \SST(k')$, we denote by $\Sp^T$ the self-dual simple $R^{\Lambda_k}$-module. 
Moreover, $\Sp^T$ is  real since it is a determinantial module \cite[Proposition 4.2]{KKKP17}.
The simple module $\Sp^T$ can be realized in terms of standard tableaux as follows \cite[Section 5]{BK09}.
Let $T \in \SST(k')$ and set $\xi_T := ( t_k-k, t_{k-1}-k+1, \ldots, t_1-1   ) $ and $m := |\xi_T|$, where $t_a$ is the entry in the $a$-th box of $T$ from the top. 
For $ S \in \ST(\xi_T)$ and a box $ b = (p,q) \in S $, we define $\res(b) = q-p+k$. Then we have the {\it residue sequence} of $S$  
$$
\res(S) = (\res(b_m), \res(b_{m-1}),  \ldots , \res(b_1))  \in I^{\Lambda_k - \wt(T)} 
$$ 
where $b_k$ is the box of $S$ with entry $k$.
We set 
$$
\Sp^T := \bigoplus_{S \in \ST(\xi_T)} \bR S
$$
 and define the action of the quiver Hecke algebra by 
\begin{align*}
x_i S = 0, \quad
\tau_j S  = 
\left\{
  \begin{array}{ll}
  s_jS & \text{ if $s_jS$ is standard}, \\
  0 &  \text{ otherwise}
  \end{array}
\right.
\quad
e(\nu) S  = 
\left\{
  \begin{array}{ll}
  S & \text{ if } \nu = \res(S), \\
  0 &  \text{ otherwise,}
  \end{array}
\right.
\end{align*}
where $s_jS$ is the tableau obtained from $S$ by exchanging the entries $j$ and $j+1$.
It is easy to check that $\Sp^T$ is a self-dual simple $R^{\Lambda_k}$-module and $\qch \Sp^T = \sum_{S \in \ST(\xi_T)} \res(S) $.

\begin{thm} \label{Thm: app1}
 Let $\lambda $ be a two-columns Young diagram with $\ell(\lambda)<n$,  and $T \in \SST(\lambda)$. 
We set $T_k$ to be the $k$-th column of $T$ for $k=1,2$, 
and write  $\sigma(T_2 \otimes T_1) = (T_2' \otimes T_1')$,
where $\sigma$ is given in $\eqref{Eq: cRmatrix}$.
We set $\beta_k = \Lambda_{|\sh(T_k)|} - \wt(T_k)$ and $\beta_k' = \Lambda_{|\sh(T_k')|} - \wt(T_k')$ for $k=1,2$.
\begin{enumerate}
\item We have 
$$
\Hom_{R\gmod}( q^d \Sp^{T_2} \conv \Sp^{T_1}, \Sp^{T_1'} \conv \Sp^{T_2'}  ) = \bR,
$$
where $d =  (\beta_1, \Lambda_{|\sh(T_2)|}) + (\beta_1', \Lambda_{ |\sh(T_2')| }) - (\beta_1', \beta_2')$.
\item If $T_1 = T_2'$ and $T_2 = T_1'$, then
$\Sp^{T_1}$ and $\Sp^{T_2}$ strongly commute.
\end{enumerate}
\end{thm}
\begin{proof}
It follows from Theorem \ref{Thm: sR} and Corollary \ref{Cor: app2}.
\end{proof}

\begin{Ex} We suppose that $n=4$. Let $\lambda = (2,1)$ and 
$$ T = \begin{tabular}{|c|c|}
     \hline
     2 & 4  \\
     \hline
     3 \\
     \cline{1-1}
   \end{tabular}.
$$
Then, $T \in \SST(\lambda)$, $\mu := \lambda' = (2,1)$, and
$$
T_1 = \begin{tabular}{|c|}
     \hline
     2   \\
     \hline
     3 \\
     \cline{1-1}
   \end{tabular}\ , 
   \qquad 
   T_2 = \begin{tabular}{|c|}
     \hline
     4   \\
     \cline{1-1}
   \end{tabular}\ .
$$
From the bijection $\eqref{Eq: b}$, we have $ \te \mathsf{b}( T_2 \otimes T_1) = \te (10 \otimes 01 \otimes 01 \otimes 00) = 10 \otimes 01 \otimes 10 \otimes 00 $, which
tells
 $\sigma(T_2 \otimes T_1) =   T_2' \otimes T_1' $, where 
$$
   T_1' = \begin{tabular}{|c|}
     \hline
     3   \\
     \cline{1-1}
   \end{tabular}\ , 
   \qquad 
T_2' = \begin{tabular}{|c|}
     \hline
     2   \\
     \hline
     4 \\
     \cline{1-1}
   \end{tabular}.
$$
Thus, $\xi_{T_1} = (1,1)$, $\xi_{T_2} = (3)$, $\xi_{T_1'} = (2)$, and $\xi_{T_2'} = (2,1)$, which tells 
$$
\qch \Sp^{T_1} = (1,2), \ \qch \Sp^{T_2} = (3,2,1), \ \qch \Sp^{T_1'} = (2,1), \ \qch \Sp^{T_2'} = (3,1,2)+(1,3,2). 
$$
By Theorem \ref{Thm: app1}, we obtain $d = (\alpha_1+\alpha_2, \Lambda_1) + (\alpha_1 + \alpha_2, \Lambda_2) - (\alpha_1+\alpha_2, \alpha_1+\alpha_2+\alpha_3) = 1 $ and
\begin{align} \label{Eq: Ex1}
\Hom_{R\gmod}( q\Sp^{T_2} \conv \Sp^{T_1}, \Sp^{T_1'} \conv \Sp^{T_2'}  ) = \bR.
\end{align}

We may also construct directly a nonzero homomorphism in $\eqref{Eq: Ex1}$.
Let us consider the following chain of nonzero homomorphisms:
\begin{align*}
 q  L(3) \conv \Sp^{T_1'} \conv \Sp^{T_1} \buildrel r_1 \conv \mathrm{id} \over \longrightarrow \Sp^{T_1'} \conv L(3) \conv \Sp^{T_1} 
 \twoheadrightarrow \Sp^{T_1'} \conv L(3) \hconv \Sp^{T_1} , 
\end{align*}
where $r_1$ is the $R$-matrix  with $\deg(r_1) = 1$ and $L(3)$ is the self-dual simple module of  $R(\alpha_3)$.  
It is easy to check that the composition of the above homomorphisms  is nonzero and 
\begin{align} \label{Eq: Ex2}
\Im r_1 \simeq  q  L(3) \hconv \Sp^{T_1'} \simeq  q  \Sp^{T_2} , \quad 
  L(3) \hconv \Sp^{T_1} \simeq  \Sp^{T_2'} .
\end{align}
Under the identification $\eqref{Eq: Ex2}$, restricting the projection to $\Im (r_1 \conv \mathrm{id} ) $, 
we have a nonzero homomorphism 
$$
 q \Sp^{T_2} \conv \Sp^{T_1} \longrightarrow \Sp^{T_1'} \conv \Sp^{T_2'}.
$$
Moreover, one can prove directly that $\Sp^{T_2} \conv \Sp^{T_1}$ is simple, which implies $\eqref{Eq: Ex1}$.
\end{Ex}

\begin{Ex} Suppose that $n=5$. Let $\lambda = (2,2,1,1)$ and 
$$ T = \begin{tabular}{|c|c|}
     \hline
     1 & 3  \\
     \hline
     3 & 5  \\
     \hline
     4   \\
     \cline{1-1}
     5 \\
     \cline{1-1}
   \end{tabular}.
$$
Then $\mu := \lambda' = (4,2)$, and
$$
T_1 = \begin{tabular}{|c|}
     \hline
     1   \\
     \hline
     3   \\
     \hline
     4  \\
     \hline
     5 \\
     \cline{1-1}
   \end{tabular}\ , 
   \qquad 
   T_2 = \begin{tabular}{|c|}
     \hline
      3   \\
     \hline
     5   \\
     \cline{1-1}
   \end{tabular}\ .
$$
Using $\eqref{Eq: b}$, we have 
$$ 
\te^2 \mathsf{b}( T_2 \otimes T_1) = \te^2 (11 \otimes 01 \otimes 11 \otimes 00 \otimes 01 ) = 
11 \otimes 10 \otimes 11 \otimes 00 \otimes 10, 
$$
which
implies
 $\sigma(T_2 \otimes T_1) = T_1 \otimes T_2$.
 Note that $\xi_{T_1} = (1,1,1)$, $\xi_{T_2} = (3,2)$ and 
\begin{align*}
\qch \Sp^{T_1} &=  (2,3,4), \\
\qch \Sp^{T_2} &= (2,1,4,3,2)+ (2,4,1,3,2) + (2,4,3,1,2) + (4,2,1,3,2) + (4,2,3,1,2).
\end{align*}
 By Theorem \ref{Thm: app1}, 
 $$
\text{ $\Sp^{T_2} \conv \Sp^{T_1}$ is simple.}
 $$ 
\end{Ex}

\begin{Rmk}
Let 
$$
T_1 = \begin{tabular}{|c|}
     \hline
     1   \\
     \hline
     3   \\
     \hline
     4  \\
     \hline
     5 \\
     \cline{1-1}
   \end{tabular}\ , 
   \qquad 
   T_2 = \begin{tabular}{|c|}
     \hline
     2   \\
     \hline
     4   \\
     \cline{1-1}
   \end{tabular}\ .
$$
One can show that $T_2 \otimes T_1 \in C_{\Lambda_2, \Lambda_4}$ and $T_1 \otimes T_2 \in C_{\Lambda_4, \Lambda_2}$, but $ T_2 \otimes T_1 \not\simeq T_1 \otimes T_2 $.
Since 
$$ 
\te^2 \mathsf{b}( T_2 \otimes T_1) = \te^2 (01 \otimes 11 \otimes 01 \otimes 10 \otimes 01 ) = 
10 \otimes 11 \otimes 10 \otimes 10 \otimes 01, 
$$
we have  $\sigma(T_2 \otimes T_1) =  T_2' \otimes T_1' $, where 
$$
   T_1' = \begin{tabular}{|c|}
     \hline
     1   \\
     \cline{1-1}
     4   \\
     \cline{1-1}
   \end{tabular}\ , 
   \qquad 
T_2' = \begin{tabular}{|c|}
     \hline
     2   \\
     \hline
     3   \\
     \cline{1-1}
     4   \\
     \cline{1-1}
     5 \\
     \cline{1-1}
   \end{tabular}.
$$
\end{Rmk}

An algorithm was given in \cite{LF00} for computing the entries of the transition matrices between standard monomials and global basis via  the embedding 
$$
V_q(\Lambda) \longrightarrow V_q(\Lambda_1)^{\otimes m_1} \otimes_- \cdots \otimes_- V_q(\Lambda_{n-1})^{\otimes m_{n-1}},
$$
where $\Lambda = \sum_{k=1}^{n-1} m_k \Lambda_k $. 
Moreover, it turned out in \cite{Brun06} that  the entries of the transition matrix can be computed in terms of the Kazhdan-Lusztig polynomials.
Note that the entries also appear in composition multiplicities of the standard modules for the finite $W$-algebras/shifted Yangians \cite{BK08}.
Let $\lambda $ be a Young diagram with $\ell(\lambda)<n$, and $\mu = \lambda '$. We write $\mu  = (\mu_1, \ldots, \mu_r)$ and 
let $\Lambda := \sum_{k=1}^r  \Lambda_{\mu_k}$. Take $T_k \in \SST(\mu_k)$ for $k=1, \ldots, r$ and set $T := T_1 * \cdots *T_r$ to be the {\it column strict tableau} obtained by concatenating
$T_k$'s. 
Via the projection $ \pi :   V_q(\Lambda_{\mu_r}) \otimes_+ \cdots \otimes_+ V_q(\Lambda_{\mu_1})  \twoheadrightarrow  V(\Lambda) $, 
we write 
\begin{align} \label{Eq: module embedding}
\pi( G^{\up}_{\Lambda_{\mu_r}} (T_r) \otimes \cdots \otimes  G^{\up}_{\Lambda_{\mu_1}} (T_1) ) = \sum_{T' \in \SST(\Lambda) } A_{T,T'}(q) G^{\up}_{\Lambda} (T'),
\end{align} 
for some  $A_{T,T'}(q) \in \Z[q,q^{-1}]$.  
Then, $A_{T,T'}(q)$ is computed in \cite[Theorem 26]{Brun06} as follows:
\begin{align} \label{Eq: KL}
A_{T,T'}(q) = (-q)^{\ell(d_T) - \ell(d_{T'})} \sum_{z\in D_{\nu_T} \cap S_{\nu_T} d_T S_\mu } (-1)^{\ell(z) + \ell(d_{T'})} P_{z w_{|\lambda|}, d_{T'} w_{|\Lambda|}}(q^{-2}).
\end{align}
We explain the notations appeared in $\eqref{Eq: KL}$ briefly. We refer the reader to \cite{Brun06} for precise definitions and notations.  
$P_{x,y}(t)$ is the Kazhdan-Lusztig polynomials associated with the symmetric group.
For $\nu \in \Z_{>0}^r$, let $D_\nu$ be the set of {\it minimal length $S_\nu \backslash S_{|\nu|}$-coset representatives}.
For a column-strict tableau $T$, $\nu_T$ is the {\it content} of $T$,  $\gamma_T$ is the {\it column reading} of $T$ from top to bottom and from right to left, and 
$d_T$ is a unique element in $D_{\nu_T}$ such that $\gamma_T d_T^{-1}$ is weakly decreasing.  
We remark that upper crystal bases dealt in \cite{Brun06} are bases at $q=\infty$ while we use bases at $q=0$. 
Combining this with Theorem \ref{Thm: main} and Theorem \ref{Thm: app_tensor}, we have the following.

\begin{thm}  \label{Thm: KL poly}
Let $\lambda$ be a Young diagram with $\ell(\lambda)<n$ and $\mu := \lambda '$. We write $\mu  = (\mu_1, \ldots, \mu_r)$, and
set $\Lambda := \sum_{k=1}^r  \Lambda_{\mu_k}$.
For $T_k \in \SST(\mu_k)$ $(k=1, \ldots, r)$ and $T' \in \SST(\lambda)$, we have 
\begin{align*}
[\Sp^{T_r} \conv \cdots \conv \Sp^{T_1} : L(T') ]_q = q^{- \sum_{1 \le a < b \le r} (\beta_a, \Lambda_{\mu_b})} A_{T, T'}(q),
\end{align*}
where $\beta_k = \Lambda_{\mu_k}- \wt(T_k) $,  $T := T_1 * \cdots *T_r$ is the column strict tableau obtained by concatenating
$T_k$'s, and $A_{T, T'}(q)$ is given in $\eqref{Eq: KL}$.
\end{thm}

\begin{Ex} Let $n=5$, $\lambda = (4,3,1)$, $\mu = \lambda' = (3,2,2,1)$ and
$$ 
T =  \begin{tabular}{|c|c|c|c|}
     \hline
     1 & 1 &3&5 \\
     \hline
     2 & 2 &4\\
     \cline{1-3}
     3   \\
     \cline{1-1}
   \end{tabular}
   , \quad 
S =  \begin{tabular}{|c|c|c|c|}
     \hline
     1 & 1 &2&5 \\
     \hline
     2 & 3 &4\\
     \cline{1-3}
     3   \\
     \cline{1-1}
   \end{tabular}. 
$$
$T_k$ and $S_k$ are the $k$-th column of $T$ and $S$ from the left respectively.
Then we have 
$\xi_{T_1} = \xi_{T_2} = (0)$, $ \xi_{T_3}= (2,2)$, $\xi_{T_4} =(4)$, 
$ \xi_{S_1} = (0)$, $\xi_{S_2} = (1)$, $\xi_{S_3} = (2,1)$, $\xi_{S_4} = (4)$, and 
\begin{align*}
&\qch\Sp^{T_3} = ( 2,3,1,2 ) + ( 2,1,3,2 ), \ \  \qch\Sp^{T_4} = (4,3,2,1), \\
& \qch\Sp^{S_2} = (2), \ \  \qch\Sp^{S_3} = (3,1,2)+(1,3,2) ,\  \ \qch\Sp^{S_4} = (4,3,2,1).
\end{align*}
Note that $ \Sp^{T_1}$, $\Sp^{T_2}$ and $\Sp^{S_1} $ are trivial.
Using the table of the entries of the transition matrix computed in \cite[Table 1]{Brun06} or \cite[Section 5]{LF00}, 
we have 
\begin{align*}
A_{T,T'} (q) = 
\left\{
  \begin{array}{ll}
  1 & \text{ if } T' = T, \\
  0 &  \text{ otherwise,}
  \end{array}
\right.
\qquad 
A_{S,S'} (q) = 
\left\{
  \begin{array}{ll}
  q & \text{ if } S' = T, \\
  1 & \text{ if } S' = S, \\
  0 &  \text{ otherwise,}
  \end{array}
\right.
\end{align*}
Since 
\begin{align*}
\sum_{1 \le a < b \le 4} (\beta_a, \Lambda_{\mu_b}) = 1, \qquad   \sum_{1 \le a < b \le 4} (\gamma_a, \Lambda_{\mu_b}) = 2,
\end{align*}
where $ \beta_k = \Lambda_{\mu_k} - \wt(T_k)$ and $ \gamma_k = \Lambda_{\mu_k} - \wt(S_k)$,
we have 
\begin{equation*} 
\begin{aligned} 
& [\Sp^{T_4} \conv \Sp^{T_3} \conv \Sp^{T_2} \conv \Sp^{T_1}] = [\Sp^{T_4} \conv \Sp^{T_3} ] = q^{-1} [L(T)],  \\   
& [\Sp^{S_4} \conv \Sp^{S_3} \conv \Sp^{S_2} \conv \Sp^{S_1}] = [\Sp^{S_4} \conv \Sp^{S_3} \conv \Sp^{S_2}] = q^{-2} (  q[L(T)] + [L(S)]  ). 
\end{aligned}
\end{equation*}
Thus, $q (\Sp^{T_4} \conv \Sp^{T_3})$ is a self-dual simple $R$-module.
We may find directly how to embed $L(T)$ to $\Sp^{S_4} \conv \Sp^{S_3} \conv \Sp^{S_2}$ using the homomorphism 
\begin{align*}
\Sp^{S_4} \conv \Sp^{S_2} \conv \Sp^{S_3} \buildrel \mathrm{id} \conv r_2 \over \longrightarrow \Sp^{S_4} \conv \Sp^{S_3} \conv \Sp^{S_2}
\end{align*}
where $r_2$ is the $R$-matrix. 
It is easy to show that $\deg r_2 = 0$, $ \Sp^{T_4} \simeq \Sp^{S_4}$ and $ \Im r_2 \simeq \Sp^{T_3}$. 
Hence we have the embedding
\[
 q^{-1} L(T) \simeq  \Sp^{T_4} \conv \Sp^{T_3} \simeq \Im (\mathrm{id} \conv r_2) \rightarrowtail \Sp^{S_4} \conv \Sp^{S_3} \conv \Sp^{S_2}.
\]

\end{Ex}

\vskip 2em


\bibliographystyle{amsplain}

\begin{thebibliography}{99}


\bibitem{APS16}
S.~Ariki, E.~Park and L.~Speyer,
\emph{Specht modules for quiver Hecke algebras of type $C$}, arXiv:1703.06425.


\bibitem{Brun06}
J.~Brundan, \emph{Dual canonical bases and Kazhdan-Lusztig polynomials},    
J. Algebra \textbf{306} (2006), no. 1, 17-–46. 




\bibitem{BK08} J.~Brundan and A.~Kleshchev,
\emph{Representations of shifted Yangians and finite $W$-algebras},       
Mem. Amer. Math. Soc. \textbf{196} (2008), no. 918.



\bibitem{BK09}
\bysame, \emph{Blocks of cyclotomic Hecke algebras and Khovanov-Lauda algebras}, Invent. Math. \textbf{178} (2009), 451--484.

\bibitem{HK02} J.~Hong and S.-J.~Kang,
\emph{Introduction to Quantum Groups and Crystal Bases},
Graduate Studies in Mathematics, \textbf{42}. American Mathematical Society, Providence, RI, 2002.



\bibitem{KK11}
S.-J. Kang and M. Kashiwara, \emph{Categorification of Highest Weight Modules via Khovanov-Lauda-Rouquier Algebras},
 Invent. Math. \textbf{190} (2012), no. 3, 699-742.
 
 \bibitem{KKKO15a}
S.-J. Kang, M. Kashiwara,  M. Kim  and   S.-j. Oh,
\newblock{\em Simplicity of heads and socles of tensor products},
Compos. Math. \textbf{151} (2015), no. 2, 377--396.


  \bibitem{KKKO17}
S.-J. Kang, M. Kashiwara, M. Kim and S.-j.~Oh, \emph{Monoidal categorification of cluster algebras},
J. Amer. Math. Soc. \textbf{31} (2018), no. 2, 349--426. 



\bibitem{Kas90}
M. Kashiwara, \emph{Crystalizing the q-analogue of universal enveloping algebras},
Comm. Math. Phys. \textbf{133} (1990), no.~2, 249--260.


\bibitem{Kas91}
\bysame, \emph{On crystal bases of the $Q$-analogue of universal enveloping algebras},
 Duke. Math. J. \textbf{154} (1991), 265--341.


\bibitem{Kas93}
\bysame, \emph{Global crystal bases of quantum groups},
 Duke. Math. J. \textbf{69} (1993), 455--485.


 \bibitem{KKKP17}
M. Kashiwara, M. Kim, S.-j.~Oh and E. Park, \emph{Monoidal categories associated with strata of flag manifolds},
Adv. Math. \textbf{328} (2018), 959--1009. 


 

 \bibitem{KN94}
M.~Kashiwara and T.~Nakashima, \emph{Crystal graphs for representations of the q-analogue of classical Lie algebras},   
J. Algebra \textbf{165} (1994), no. 2, 295–-345. 
   



\bibitem{KL09}
M.~Khovanov and A. D.~Lauda, \emph{A diagrammatic approach to categorification of quantum groups
  {I}}, Represent. Theory \textbf{13} (2009), 309-347.

\bibitem{KL11}
\bysame, \emph{A diagrammatic approach to categorification of quantum groups
  {II}}, Trans. Amer. Math. Soc. \textbf{363} (2011), no.~5, 2685-2700.


\bibitem{Kimura12}
Y.~Kimura, \emph{Quantum unipotent subgroup and dual canonical basis}, Kyoto J. Math. \textbf{52} (2012), no.~2, 277-331.


\bibitem{KR10}
A.~Kleshchev and A.~Ram, \emph{Homogeneous representations of Khovanov-Lauda algebras}, 
J. Eur. Math. Soc. (JEMS) \textbf{12} (2010), no. 5, 1293–--1306. 




\bibitem{LNT03}
B.~Leclerc,  M.~Nazarov, J.-Y.~Thibon,
\emph{Induced representations of affine Hecke algebras and canonical bases of quantum groups}, In Studies in memory of Issai Schur,  
Progr. Math., \textbf{210}, Birkhäuser (2003). 


\bibitem{LV11}
A. D.~Lauda and M.~Vazirani, \emph{Crystals from categorified quantum groups},
Adv. Math. \textbf{228} (2011), no.~2, 803-861.

\bibitem{LF00}
B.~Leclerc and P.~Toffin, \emph{A simple algorithm for computing the global crystal basis of an irreducible $U_q(\mathfrak{sl}_n)$-module},   
Internat. J. Algebra Comput. \textbf{10} (2000), no. 2, 191–-208. 




\bibitem{NY97}
A.~Nakayashiki and Y.~Yamada, \emph{Kostka polynomials and energy functions in solvable lattice models},   
Selecta Math. (N.S.) \textbf{3} (1997), no. 4, 547-–599. 


\bibitem{R08}
R.~Rouquier, \emph{2 {K}ac-{M}oody algebras}, arXiv:0812.5023 (2008).

\bibitem{R11}
\bysame, {\em Quiver Hecke algebras and 2-Lie algebras},
Algebra Colloq. {\bf 19} (2012), no. 2, 359--410.


\bibitem{VV11}
M.~Varagnolo and E.~Vasserot, \emph{Canonical bases and {K}{L}{R}  algebras}, J. Reine Angew. Math. \textbf{659} (2011), 67--100.


\bibitem{Web17}
B.~Webster, \emph{Knot invariants and higher representation theory}, Mem. Amer. Math. Soc. \textbf{250}, (2017),  no.~1191, 141pp.



\end{thebibliography}


\end{document}